\documentclass[11pt,a4paper]{article}

\usepackage[utf8]{inputenc}
\usepackage[T1]{fontenc}
\usepackage[english]{babel}
\usepackage[sc]{mathpazo}          
\usepackage{microtype}

\usepackage{amsmath}
\usepackage{amsthm}
\usepackage{amsfonts}
\usepackage{amssymb}
\usepackage{bm}
\usepackage{graphicx}
\usepackage{caption}
\usepackage{subcaption}
\usepackage{booktabs}
\usepackage{multirow}
\usepackage{siunitx}
\usepackage{enumitem}

\usepackage[a4paper,
            top=2.7cm, bottom=3.0cm,
            left=2.5cm, right=2.5cm,
            headsep=0.9cm]{geometry}

\usepackage{xcolor}
\usepackage[skins]{tcolorbox}
\definecolor{accent}{RGB}{27, 54, 93}     
\definecolor{rulegray}{RGB}{110, 110, 110}
\definecolor{citegreen}{RGB}{0, 82, 82}

\usepackage{titlesec}
\titleformat{\section}
  {\Large\bfseries}
  {\textcolor{accent}{\thesection}}{0.85em}{}
\titleformat{\subsection}
  {\large\bfseries}
  {\textcolor{accent}{\thesubsection}}{0.85em}{}
\titleformat{\subsubsection}
  {\normalsize\bfseries}
  {\textcolor{accent}{\thesubsubsection}}{0.85em}{}
\titlespacing*{\section}{0pt}{2.1ex plus .6ex minus .2ex}{1.3ex plus .3ex}
\titlespacing*{\subsection}{0pt}{1.8ex plus .5ex minus .2ex}{1.0ex plus .2ex}

\usepackage{fancyhdr}
\fancypagestyle{firstpage}{%
  \fancyhf{}%
  \fancyfoot[C]{\footnotesize\scshape\color{rulegray} Preprint \,\textperiodcentered\, July 2026}%
}

\newtheoremstyle{accentplain}%
  {0.9em}{0.9em}%
  {\itshape}{}%
  {\bfseries\color{accent}}{.}{ }{}
\newtheoremstyle{accentdefinition}%
  {0.9em}{0.9em}%
  {\normalfont}{}%
  {\bfseries\color{accent}}{.}{ }{}
\theoremstyle{accentplain}
\newtheorem{theorem}{Theorem}[section]

\newtheorem{proposition}[theorem]{Proposition}
\newtheorem{remark}[theorem]{Remark}

\theoremstyle{accentdefinition}
\newtheorem{test}{Test}[section]

\newcommand{\PP}{{\mathbb P}}

\newcommand{\NN}{\mathbb{N}}

\newcommand{\R}{\mathcal{R}}
\newcommand{\C}{\mathcal{C}}
\newcommand{\V}{\mathcal{V}}

\newcommand{\bigO}{\mathcal{O}}
\newcommand{\m}{{m^{\hspace{.05cm}\mbox{\large $\bot$}}}}

\DeclareMathOperator{\diag}{diag}

\usepackage[numbers,sort&compress]{natbib}
\newcommand{\doi}[1]{doi:\,\href{https://doi.org/#1}{\nolinkurl{#1}}}
\usepackage[colorlinks=true,
            linkcolor=accent,
            citecolor=citegreen,
            urlcolor=accent]{hyperref}
\usepackage{orcidlink}

\newcommand{\affil}[2]{\textsuperscript{\normalfont#1}\,#2}
\newcommand{\authormark}[1]{\textsuperscript{\normalfont#1}}

\begin{document}

\pagestyle{fancy}
\thispagestyle{firstpage}

\begin{center}
    {\color{accent}\rule{\textwidth}{1.1pt}}\\[-2.5pt]
    {\color{accent}\rule{\textwidth}{0.4pt}}

    \vspace{1.6em}
    {\LARGE\bfseries A Rational Discrete Collocation Method\\[0.35em]
     for Second Kind Fredholm Equations\par}

    \vspace{1.7em}
    {\large
     Domenico Mezzanotte\,\orcidlink{0000-0001-5154-6538}\,\authormark{a,b} \qquad
     Donatella Occorsio\,\orcidlink{0000-0001-9446-4452}\,\authormark{a,b}\\[0.45em]
     Mario Pezzella\,\orcidlink{0000-0002-1869-945X}\,\authormark{b,c,\textasteriskcentered} \qquad
     Woula Themistoclakis\,\orcidlink{0000-0002-6185-1154}\,\authormark{b}\par}

    \vspace{1.2em}
    {\small\itshape
     \affil{a}{Department of Basic and Applied Sciences, University of Basilicata,\\
               Via dell'Ateneo Lucano 10, 85100 Potenza, Italy}\\[0.3em]
     \affil{b}{Institute for Applied Mathematics ``Mauro Picone'', National Research Council of Italy,\\
               Via Pietro Castellino 111, 80131 Naples, Italy}\\[0.3em]
     \affil{c}{Department of Mathematics and Applications ``Renato Caccioppoli'',\\
               University of Naples Federico II, Via Cintia, 80126 Naples, Italy}\par}

    \vspace{0.9em}
    {\textasteriskcentered\;Corresponding author:\;
     {\ttfamily\href{mailto:mario.pezzella@cnr.it}{mario.pezzella@cnr.it}}
     \,\textperiodcentered\,
     {\ttfamily\href{mailto:mario.pezzella@unina.it}{mario.pezzella@unina.it}}\par}

    \vspace{1.2em}
    {\color{accent}\rule{0.35\textwidth}{0.4pt}}
\end{center}

\vspace{0.6em}

\begin{center}
\begin{tcolorbox}[enhanced, width=0.94\textwidth,
                  colback=accent!5, colframe=accent!5,
                  boxrule=0pt, sharp corners,
                  borderline west={1.6pt}{0pt}{accent},
                  left=10pt, right=10pt, top=9pt, bottom=9pt]
    {\small
    \noindent{\bfseries\color{accent} Abstract.}
    In this work we present a novel discrete collocation method for the numerical solution of Fredholm integral equations of the second kind in the space of continuous functions equipped with the uniform norm. The method is based on a rational interpolation scheme recently developed within the general framework of reproducing kernel Hilbert spaces. This rational approximation has no real poles, interpolates the target function at arbitrary Jacobi nodes and exhibits uniformly bounded Lebesgue constants. Moreover, it converges uniformly for all continuous functions at a rate at least equal to that of the best uniform polynomial approximation. These interesting properties are inherited by the resulting numerical method, for which stability, convergence and good conditioning are established under minimal assumptions on the integral kernel. A series of numerical experiments confirm the theoretical findings and indicate that, in the presence of particularly challenging kernels, the proposed approach provides a robust and effective alternative to Nystr\"om-type methods.

    \vspace{0.9em}
    \noindent{\bfseries\color{accent} Keywords.}
    Fredholm integral equations \,\textperiodcentered\, Rational interpolation \,\textperiodcentered\, de la Vall\'ee Poussin interpolation \,\textperiodcentered\, Discrete collocation \,\textperiodcentered\, Near--best approximation.
    }
\end{tcolorbox}
\end{center}

\vspace{0.8em}


\section{Introduction}\label{sec:Intro}
In \cite{Th-Barel}, a novel linear rational approximation process on $[-1,1]$ has been introduced as an example within the general setting of reproducing kernel Hilbert spaces. 
More precisely, by selecting Darboux kernels associated with an arbitrary Jacobi weight, a rational approximant $R_n^mf$ is built, which is free of real poles and interpolates the target function $f$ at the corresponding Jacobi zeros. This rational interpolant has the peculiarity of depending on two integer parameters $0<m<n$ where $n$ determines the number $n+1$ of nodes and $m$ is a free parameter, which can be appropriately tuned to improve the local approximation properties, especially in the presence of isolated singularities. Moreover,  as $m\sim n \to\infty$, $R_n^mf$ exhibits uniformly bounded Lebesgue constants and uniformly converges to $f$ in the space of continuous functions $C^0=C[-1,1]$, in contrast to classical Lagrange polynomial interpolation at the same Jacobi zeros. We recall that de la Vallée Poussin type approximation \cite{Th-1999, Woula_Uniform, Woula_Weighted, TB-GVP} also exhibits the aforementioned properties, but only for certain Jacobi weights, while the rational interpolant $R_n^m$ satisfies them with no restrictions on the weight.

The objective of this work consists of testing the performance of such rational interpolation process when applied  to the numerical solution of the second kind Fredholm Integral Equations (FIEs) of the following form
\begin{equation}\label{FIE}
    f(x)-\int_{-1}^1h(x,y)f(y) \, w(y) dy=g(x), \qquad |x|\leq1,
\end{equation}
where $g$ and $h$ are known functions, $w(y) =v^{\alpha,\beta}(y):= (1-y)^\alpha(1+y)^\beta$ is a given Jacobi weight of parameters $\alpha,\beta>-1,$ and $f$ denotes the unknown function.

Our investigation is conducted assuming that the homogeneous equation corresponding to \eqref{FIE} has only the trivial solution, and supposing that the integral kernel $h$ is a bounded function satisfying the following conditions
\begin{align}
\label{hp1-h}
    & \int_{-1}^1 |h(x,y)|w(y)dy<\infty, \qquad \forall x\in [-1,1],\\
    \label{hp2-h}
    & \lim_{x^\prime\to x}\int_{-1}^1 |h(x^\prime,y)-h(x,y)|w(y)dy=0, \qquad \forall x\in [-1,1],
\end{align}
As discussed in \cite{Gra-Sloan}, conditions \eqref{hp1-h}--\eqref{hp2-h} guarantee the compactness of the operator $H:C^0\to C^0$ defined by
\begin{equation}\label{H}
    Hf(x)=\int_{-1}^1 h(x,y)f(y)w(y)dy, \qquad |x|\leq 1,
\end{equation}
so that, by Fredholm alternative,  for all $g\in C^0,$ there exists a unique solution $f^*\in C^0$ to \eqref{FIE}. 

The numerical approximation of such solution has attracted significant attention in the literature, resulting in a wide variety of methods (see, for instance, \cite{Atkinson_1997,Kress_2014}). One of the most widely used approaches is the Nyström method, whose flexibility has led to numerous extensions and variants in recent years, including averaged Nyström interpolants \cite{Luisa_Avg}, multidimensional formulations \cite{De_Bonis_2008,Laguardia_2023,Multi_Dim_Nystrom}, and schemes for highly oscillatory kernels \cite{Fermo2023NystromVF,Oscillating_Nystrom}. In its classical formulation, the Nystr\"om method applies a quadrature rule to the integral term in \eqref{FIE} and computes the approximate solution by the Nystr\"om interpolant based on the solution of a linear system achieved by collocation at the quadrature nodes. As is well--known, the error of the resulting approximation is governed by the accuracy of the underlying quadrature rule, whose choice plays a crucial role in determining the overall performance of the method.

The Nystr\"om method based on Gauss–Jacobi quadrature rule is known to provide high accuracy for smooth kernels and solutions \cite{Gautschi}. Recently, the Nystr\"om approach has been extended to include anti Gauss \cite{Patricia,Laurie_1996}, averaged Gauss and weighted averaged Gauss \cite{Luisa_Avg,SpalevicI,SpalevicII} and spline-based \cite{Spline_Nystrom} quadrature rules, achieving better performance. However, the reduced regularity of the solution, the presence of singularities in the known terms, as well as highly oscillatory kernels may significantly deteriorate the convergence rate of the previous methods, which provide a poor approximation in such \emph{critical cases}. The present study is motivated by the need for a robust and effective alternative for addressing these challenging scenarios.

Using the classical projection-collocation scheme, the application of $R_n^m$ yields a Rational Collocation (RC) method which is stable, well--conditioned and convergent. Nevertheless, the associated linear system involves the exact evaluation of certain integrals, whose computation may be infeasible or prohibitively demanding.  To overcome this problem, we construct a Rational Discrete Collocation (RDC) method which, besides $R_n^m$, also employs the de la Vall\'ee Poussin interpolation polynomial based on the Chebyshev nodes of the first kind \cite{Woula_Donatella_Cheb_1D, OT-DRNA21, Woula_Donatella}.  
By combining uniform convergence, bounded Lebesgue constants, and tunable local approximation via the parameter $m$, we prove that RDC method provides stable and uniformly convergent approximations under minimal assumptions.  Moreover, several numerical tests demonstrate its superior performance, compared to Nystr\"om type methods, in the presence of isolated singularities or highly oscillatory kernels. 

The paper is organized as follows. Section \ref{sec:Rational_Interpolation} reviews the main results on rational interpolation at Jacobi zeros. Section \ref{sec:RC_Method} introduces the RC method and discusses its convergence, stability and limitations. A Modified Nystr\"om (MN) method based on the rational operator is there introduced, as well. To overcome the drawbacks of both the RC and MN schemes, the RDC framework is presented in Section \ref{sec:RDC_Method}. Numerical examples are reported in Section \ref{sec:Numerical_Tests}. Finally, Section \ref{sec:conclusions} concludes the paper with some remarks and directions for future research.

\section{Rational interpolation at Jacobi zeros}\label{sec:Rational_Interpolation}
We recall that, given $n\in\NN_0$ and a Jacobi weight $w(y) = (1-y)^\alpha(1+y)^\beta,$ the corresponding Darboux kernel is defined as
\begin{equation*}
    k_n(x,y)=\sum_{j=0}^n p_j(x)p_j(y), \qquad x,y\in [-1,1],
\end{equation*}
where $p_j$ denotes the orthonormal Jacobi polynomial of degree $j$, with positive leading coefficient, associated with the weight $w$. As is well--known,  $k_n$ is a reproducing kernel in the space $\PP_n$ of polynomials of degree at most $n$, that is,
\begin{equation}\label{inva}
    \int_{-1}^1k_n(x,y)P(y)w(y)dy=P(x), \quad |x|\leq 1, \qquad \forall P\in\PP_n.
\end{equation}
Let $x_i=x_{n,i}(w)$,  $i=1,\ldots, n+1$, denote the zeros of $p_{n+1}$, and $\lambda_i=\lambda_{n,i}(w)=1/k_n(x_i,x_i)$ be the corresponding Christoffel numbers. By the well-known Christoffel--Darboux formula, we get
\begin{equation}\label{interp-k}
    \lambda_i k_n(x_i,x_j)=\delta_{i,j},\qquad i,j=1,\ldots, n+1,
\end{equation}
where $\delta_{i,j}$ is the Kronecker delta.
Furthermore, recalling the Gauss quadrature rule
\begin{equation}\label{Gauss}
    \int_{-1}^1 g(y)w(y)\,dy=\sum_{i=1}^{n+1}\lambda_i g(x_i), \qquad \forall g\in\PP_{2n+1},
\end{equation}
from the reproducing property \eqref{inva}, we obtain
\begin{equation}\label{inva1}
    \sum_{i=1}^{n+1}\lambda_i k_n(x,x_i)P(x_i)=P(x), \qquad |x|\le1,\quad \forall P\in\PP_n.
\end{equation}
In \cite[Example 1]{Th-Barel}, for any integers $m$ and $n$ satisfying $0<m<n$, the Darboux kernels associated with an arbitrary Jacobi weight $w$ are used to define the rational functions
\begin{align}\label{rnm}
    r_{n,i}^m(x)&=\lambda_i \frac{k_n(x_i, x)k_m(x_i,x)}{k_m(x,x)},\qquad i=1,\ldots, n+1, \qquad |x|\leq 1,\\
    \label{Rnm}
    R_n^mf(x)&=\sum_{i=1}^{n+1}f(x_i)r_{n,i}^m(x),\qquad \qquad |x|\leq 1.
\end{align}
Some examples of the basis functions in \eqref{rnm} are displayed in Figures \ref{fig:Ar_Interpol_Rational_Basis_Plot} and \ref{fig:Ar_Rational_Basis_Plot}, where their behavior is analyzed with respect each parameter's variation. In particular, we observe that increasing $n$  enhances the localization of $r_{n,i}^m(x)$ around the node $x_i$, while increasing $m$ attenuates its oscillating behavior.
\begin{figure}[htbp]
    \centering
    \includegraphics[width=0.8\linewidth]{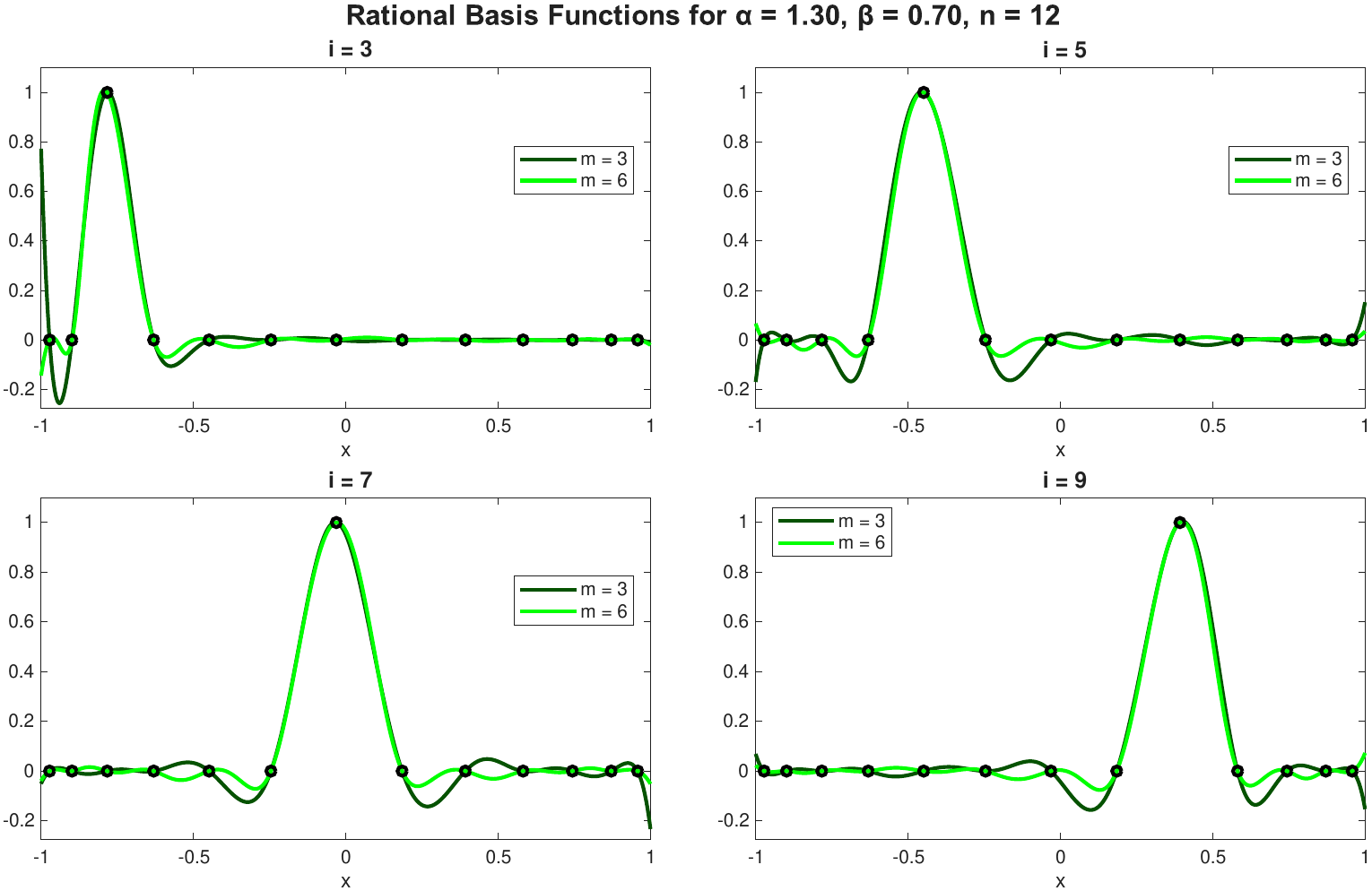}
    \caption{Rational basis functions $r_{n,i}^m(x)$ for $n=12$ and $m\in\{3,6\}$. Black circles indicate the Jacobi nodes associated with the weight $w(x)=(1-x)^{13/10}(1+x)^{7/10}$.}
  \label{fig:Ar_Interpol_Rational_Basis_Plot}
\end{figure}
\begin{figure}[htbp]
    \centering
    \includegraphics[width=0.8\linewidth]{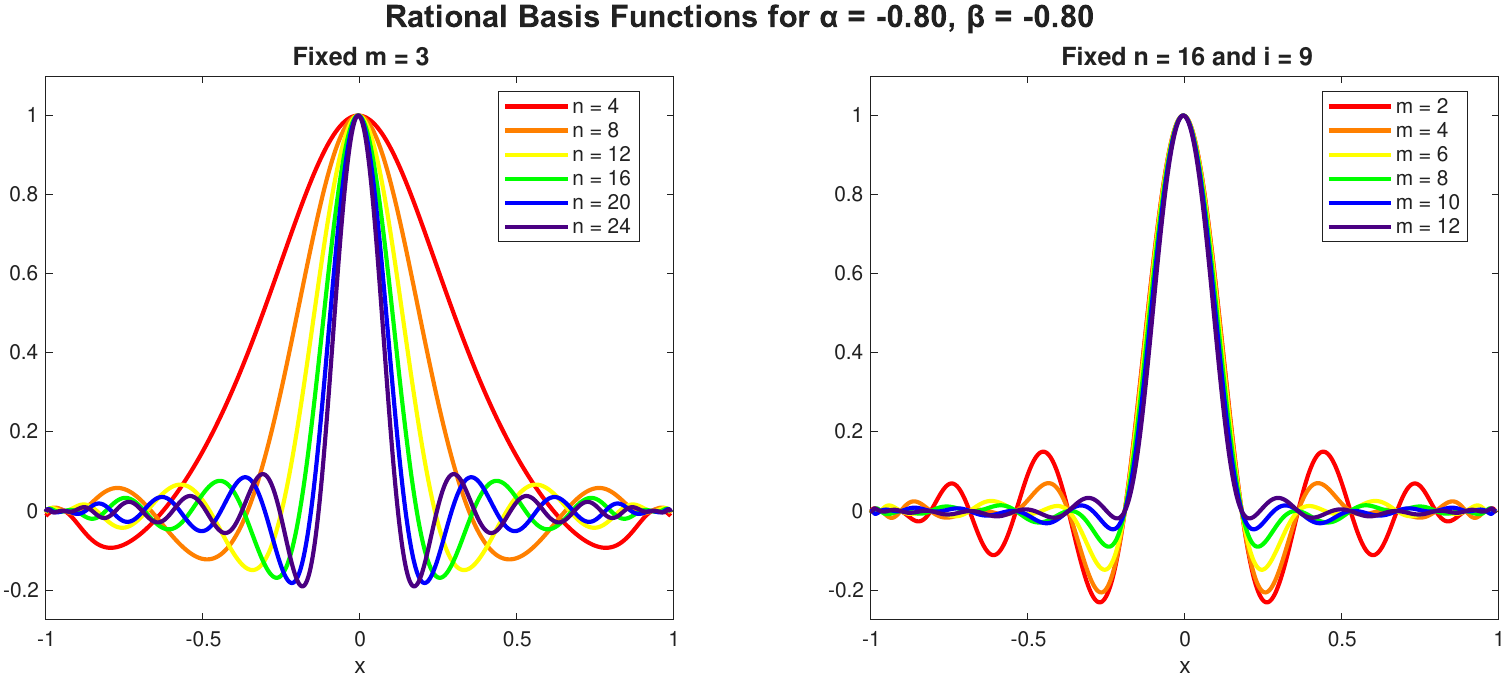}
    \caption{Rational basis functions $r_{n,i}^m(x)$ for the Jacobi weight $w(x)=(1-x)^{-4/5}(1+x)^{-4/5}$. Left: basis functions associated with the Jacobi zero closest to the origin for fixed $m=3$ and increasing $n$. Right: basis functions corresponding to the ninth Jacobi node for fixed $n=16$ and increasing values of $m$.}
  \label{fig:Ar_Rational_Basis_Plot}
\end{figure}

As displayed in Figure \ref{fig:Ar_Interpol_Rational_Basis_Plot}, we have the following interpolation property that trivially follows from \eqref{interp-k}.
\begin{proposition}\label{prop-int}
    For any Jacobi weight $w$, any pair of integers $0<m<n$, and any function $f$ defined on $(-1,1)$, we have
    \begin{align}\label{interp-r}
        r_{n,i}^m(x_j)&=\delta_{i,j}, \, \qquad i,j=1,\ldots, n+1,\\
        \label{interp}
        R_n^mf(x_j)&=f(x_j),\qquad j=1,\ldots,n+1.
    \end{align}
\end{proposition}
Furthermore, the reproducing property \eqref{inva1} directly leads to the following invariance result.
\begin{proposition}\label{prop-inva}
    For any Jacobi weights $w$ and any pair of integers  $0<m<n$, we have
    \begin{equation}\label{inva3}
    R_n^mP(x)=P(x), \qquad |x|\leq 1,\qquad \forall P\in\PP_{n-m}.
    \end{equation}
\end{proposition}

Let us introduce the rational functions space
\begin{equation}\label{R_n^m}
    \R_n^m=\operatorname{span}\{r_{n,i}^m(x), \ i=1,\ldots,n+1\}, \qquad 0<m<n.
\end{equation}
Since \eqref{interp-r} ensures the linear independence of the rational basis functions in \eqref{rnm}, then $\mbox{dim} R_n^m=n+1$. Moreover,  by virtue of \eqref{inva3},  $\PP_{n-m}\subset R_n^m$. Furthermore, the interpolation property \eqref{interp} implies that 
\begin{equation}\label{proj}
    f\in\R_n^m
    \quad\Longleftrightarrow\quad
    R_n^mf=f,
\end{equation}
and therefore the linear operator
$R_n^m:f\to R_n^mf$ defined by \eqref{Rnm} is a projection onto $\R_n^m$. Its stability is strictly related to the behavior of the Lebesgue constant
\begin{equation}\label{LC-Sum}
    \|R_n^m\|=\sup_{|x|\leq 1}\sum_{i=1}^{n+1}|r_{n,i}^m(x)|, \qquad 0<m<n,
\end{equation}
which is precisely the operator norm of the map $R_n^m:C^0\to C^0$, being $C^0=C[-1,1]$ the Banach space of continuous functions endowed with the uniform norm $\|f\|_\infty=\max_{|x|\leq 1}|f(x)|.$ 

Throughout the paper, we denote by $\C$ a positive constant that can take different values at different occurrences, writing $\C\ne \C(a,b,..)$ in the case that $\C$ is independent of the parameters $a,b,...$. Moreover, if $A,B>0$ depend on some parameters, then the notation $A\sim B$ indicates that there exists an absolute constant $\C>0$, independent of these parameters, such that
$\C^{-1}A\leq B\leq \C A$ holds.

The following proposition establishes the uniform boundedness of the Lebesgue constants in \eqref{LC-Sum} when $m\sim n$.
\begin{proposition}\label{prop-LC}
    Let $w$ be any Jacobi weight and let $n,m\in\NN$, be such that $m<n$ and $m\sim n.$ Then the projection map $R_n^m:C^0\to C^0$ is uniformly bounded with respect to $n$, i.e.
    \begin{equation}\label{LC-bound}
        \sup_{\substack{n\in\NN\\m\sim n}}\|R_n^m\|=\mathcal{O}(1).
    \end{equation}
    Moreover, for any function $f$ defined on $(-1,1)$, we have
    \begin{equation}\label{bound-R}
        \max_{1\leq i\leq n+1}|f(x_i)|\leq \|R_n^mf\|_\infty \leq \C \max_{1\leq i\leq n+1}|f(x_i)|, \qquad\C\ne\   C(n,m,f).
    \end{equation} 
\end{proposition}
\begin{proof} 
    A detailed proof of this result is given in \cite{Th-Barel}. Here, for completeness, we provide a sketch of the argument. Note that, by Cauchy–Schwarz inequality and \eqref{inva1}, we get
    \begin{equation}\label{LC}
        \|R_n^m\|\leq\sup_{|x|\leq 1}\sqrt{\frac{k_n(x,x)}{k _m(x,x)}} .
    \end{equation}
    On the other hand, by \cite[Thm. 6.3.28]{Nevai} (see also \cite[Thm. 4.7.8]{Nevai1}),  we have
    \begin{equation}\label{CF}
        k_n(x,x)\sim n \left(\sqrt{1-x}+n^{-1}\right)^{-2\alpha-1}
    \left(\sqrt{1+x}+n^{-1}\right)^{-2\beta -1},\qquad |x|\leq 1,
    \end{equation}
    and therefore there exists a constant $\C\ne \C(n,x)$ such that
    \begin{equation*}
        \sqrt{\frac{k_n(x,x)}{k _m(x,x)}}\sim \sqrt{\frac nm}\left(\frac{\sqrt{1-x}+m^{-1}}{\sqrt{1-x}+n^{-1}}\right)^{\alpha+\frac 12}\left(\frac{\sqrt{1+x}+m^{-1}}{\sqrt{1+x}+n^{-1}}\right)^{\beta +\frac 12}\leq \C\left(\frac nm\right)^{1+\max\{\alpha,\beta\}}, \quad |x|\leq 1,
    \end{equation*}
    which, by the hypothesis $n\sim m$ and \eqref{LC}, proves \eqref{LC-bound}. Furthermore, from \eqref{interp}, \eqref{LC-Sum} and \eqref{LC-bound}, 
    \begin{align*}
        \max_{1\leq i\leq n+1}|f(x_i)|&= \max_{1\leq i\leq n+1}|R_n^mf(x_i)|\leq \max_{|x|\leq 1} |R_n^mf(x)|\leq 
         \max_{|x|\leq 1} \sum_{i=1}^{n+1} |f(x_i) r_{n,i}^m(x)|\\
         &\leq \max_{1\leq i\leq n+1}|f(x_i)| \|R_n^m\|\leq \C  \max_{1\leq i\leq n+1}|f(x_i)|, \qquad \qquad \qquad \qquad \qquad \qquad \C\ne \C(n,f), 
    \end{align*}
    that yields \eqref{bound-R} and completes the proof.
\end{proof}

Let
\begin{equation*}
    E_n(f)_\infty=\inf_{P\in\PP_n}\|f-P\|_\infty,
    \qquad n\in\NN,
\end{equation*}
denote the error of best uniform approximation of $f$ by polynomials of degree at most $n$. Recall that, by the Weierstrass approximation theorem,
\begin{equation*}
    \lim_{n\to\infty}E_n(f)_\infty=0, \qquad \forall f\in C^0,
\end{equation*}
while the decay rate is determined by the smoothness of $f$, as detailed in \cite{DT}.

As a consequence of Proposition \ref{prop-LC} and \eqref{inva3}, we derive the following convergence result. 

\begin{proposition}\label{prop-conv}
    Under the assumptions of Proposition \ref{prop-LC}, for all $f\in C^0$ the following  error estimate holds
    \begin{equation}\label{err-R}
        \|R_n^mf-f\|_\infty
        \leq \C E_{n-m}(f)_\infty,
        \qquad \C\ne\C(n,m,f)
    \end{equation}
    Furthermore, if the free parameter $m$ is chosen such that $m\sim n\sim (n-m)$ holds for all sufficiently large $n\in\NN,$ then
    \begin{equation}\label{conv-R}
        \lim_{n\to\infty}\|R_n^mf-f\|_\infty=0, \qquad \forall f\in C^0,
    \end{equation}
    with the same convergence rate as the best approximation error $E_n(f)_\infty$.
\end{proposition}
\begin{proof}
Let $P^*\in\PP_{n-m}$ be  such that $E_{n-m}(f)_\infty=\|f-P^*\|_\infty.$ By \eqref{inva3} and \eqref{LC-bound}, we obtain
\begin{equation*}
    \|R_n^mf-f\|_\infty=\|R_n^m(f-P^*)-(f-P^*)\|_\infty\leq (\|R_n^m\|+1)\|f-P^*\|_\infty\leq \C E_{n-m}(f)_\infty,
\end{equation*}
which proves \eqref{err-R}. 
\newline
As $n\to\infty$, the additional assumption $n\sim(n-m)$, ensures that also $(n-m)\to\infty$, and that $E_{n-m}(f)_\infty\to 0,$
with the same decay rate as $E_n(f)_\infty$. The convergence result
\eqref{conv-R} then follows from \eqref{err-R}.
\end{proof}

We remark that the assumption $m \sim n \sim (n-m)$ is satisfied, for instance, by choosing
\begin{equation}\label{eq:n_theta}
    m=\lfloor \theta n\rfloor, 
    \qquad \theta\in(0,1).
\end{equation}
This assumption will be maintained throughout the numerical experiments.

The boundedness of the Lebesgue constant established in Proposition~\ref{prop-LC}, together with the convergence result of Proposition~\ref{prop-conv}, indicates that the rational interpolant $R_n^m$ possesses enhanced approximation properties with respect to the classical Lagrange interpolation, especially in the approximation of less regular functions. To illustrate this behavior, we consider the function
\begin{equation*}
    f(x) = \sin(6x)+\operatorname{sign}(\sin(x+\operatorname{exp}(x))), \qquad |x|\le 1,
\end{equation*}
and compare the classical Lagrange interpolant $L_n$ with the rational interpolant \(R_n^m\) at the same $n+1$ Jacobi zeros associated with $w(y)=(1-y)^{0.3}(1+y)^{0.6}$, taking $n=100$ and \(m\) computed as in \eqref{eq:n_theta} for different values of the parameter \(\theta\).
Figure~\ref{fig:Ar_Interpolator_Comaprison}  shows that the rational interpolant significantly reduces the oscillatory behavior near the singularities of the function. This effect is observed for all admissible values of \(\theta\), and becomes more pronounced as \(\theta\) increases. In particular, the amplitude of the spurious oscillations progressively decreases for larger values of \(\theta\), highlighting the stabilizing effect induced by the rational basis. It is worth noting that the largest errors of the Lagrange interpolant are attained at the endpoints, consistently with the non-convergence of the Lagrange interpolation operator for the considered weight ($w$ has exponents $\alpha, \beta >-1/2$).
\begin{figure}[htbp]
    \centering
    \includegraphics[width=0.75\linewidth]{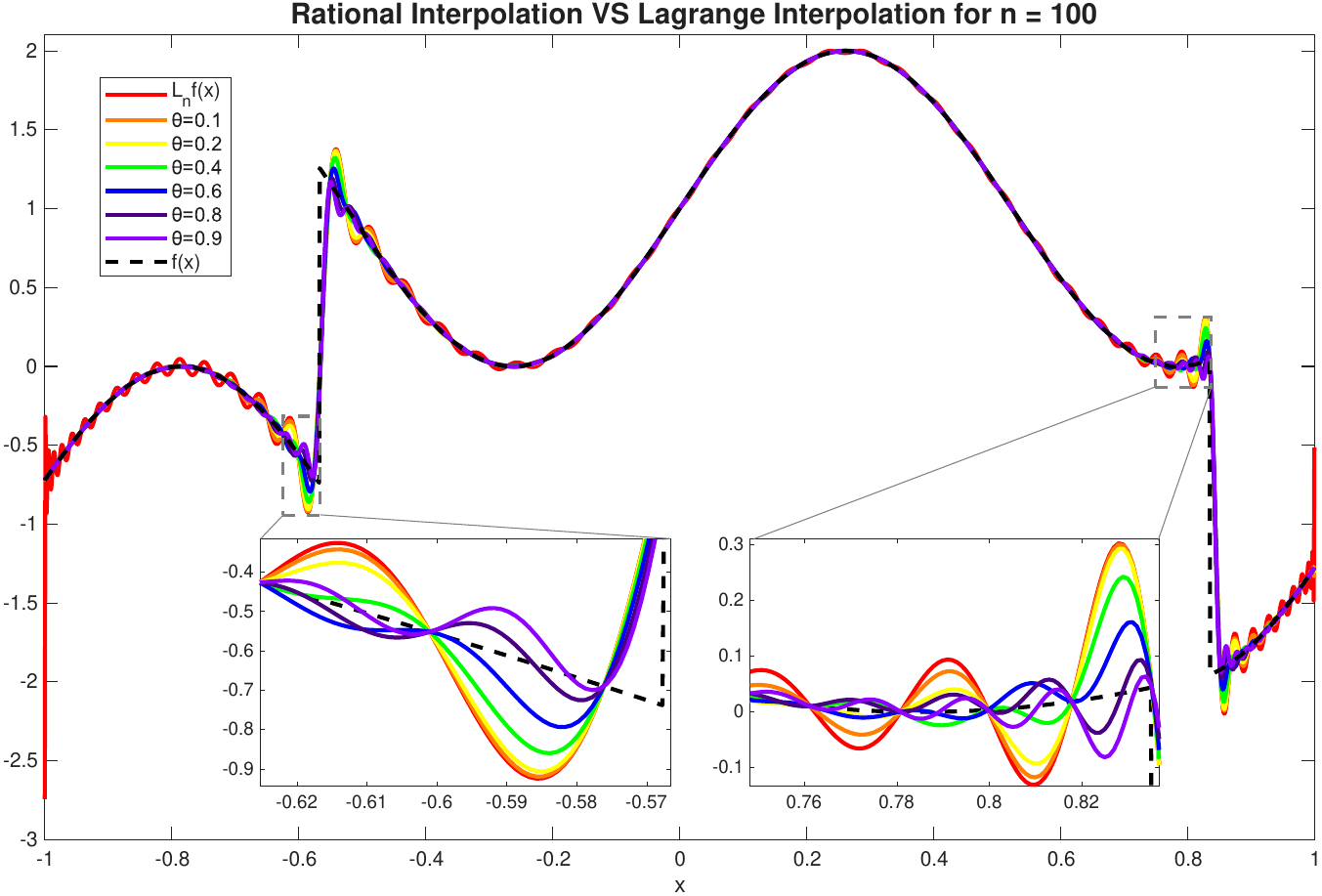}
    \caption{Comparison between the classical Lagrange interpolant $L_nf$ and the rational interpolant \(R_n^mf\) for different values of the parameter \(\theta\). Here, the weight is $w(y)=(1-y)^{0.3}(1+y)^{0.6}$ and \(m\) is computed as in \eqref{eq:n_theta}.}
  \label{fig:Ar_Interpolator_Comaprison}
\end{figure}

\section{The associated projection method}\label{sec:RC_Method}
Here, we introduce the Rational Collocation (RC) method that is the classical collocation method based on the rational
interpolating projection $R_n^m$ previously defined.

In operator form, equation \eqref{FIE} can be written as
\begin{equation}\label{FIE-op}
    (I-H)f(x)=g(x), \qquad |x|\leq 1,
\end{equation}
where $I$ denotes the identity operator and $H$ is defined in \eqref{H}.
\newline
Applying the operator $R_n^m$ defined in \eqref{Rnm} to both sides of \eqref{FIE-op}, we project the equation  onto the rational space $\R_n^m$ and obtain the approximate equation
\begin{equation}\label{FIE-project}
    (I-H_n^m) f(x)=g_n^m(x), \qquad |x|\leq 1,
\end{equation}
where, for simplicity, the same notation $f$ is retained for the unknown function, and we set
\begin{align}\label{gnm}
g_n^m(x)&:=R_n^mg(x)=\sum_{i=1}^{n+1}g(x_i)r_{n,i}^m(x), \qquad \qquad \quad \!\! |x|\leq 1,  \\
    \label{Hnm}
    H_n^mf(x)&:= R_n^m(Hf)(x)=\sum_{i=1}^{n+1}Hf(x_i)r_{n,i}^m(x), \qquad |x|\leq 1.
\end{align}

In what follows, given a linear operator $A:C^0\to C^0$, we denote by $\|A\|:=\sup_{\|f\|_\infty\leq 1}\|Af\|_\infty$  its operator norm and  by $\kappa(A):=\|A\|\|A^{-1}\|$ its condition number. The following theorem follows from standard arguments.
\begin{theorem}\label{th-collo}
	Assume that $\operatorname{ker}(I-H)=\{0\}$ and \eqref{hp1-h}-\eqref{hp2-h} hold. Moreover, for any $g\in C^0$, let $f^*\in C^0$ be the unique solution of \eqref{FIE}. Then, for all sufficiently large $n\in\NN$ and for all $0<m<n$ such that $m\sim n\sim(n-m)$, the  approximate equation \eqref{FIE-project} admits a unique solution $f_n^m$ which belongs to $\R_n^m$.  Moreover, 
    \begin{equation}\label{lim-col}
        \lim_{n\to\infty}\|f^*-f_n^m\|_\infty=0,
    \end{equation}
    with the error estimate
    \begin{equation}\label{er-col}
        \|f^*-f_n^m\|_\infty
        \leq \C\big(
        E_{n-m}(g)_\infty+E_{n-m}(Hf^*)_\infty
        \big),\qquad \C\ne\C(n,m,g, f^*).
    \end{equation}
Finally, as regards the conditioning of equations  \eqref{FIE-project} and \eqref{FIE-op}, we have
    \begin{equation}\label{lim-cond}
        \lim_{n\to\infty}\kappa(I-H_n^m)=\kappa(I-H).
    \end{equation}
\end{theorem}
\begin{proof}
First, observe that for all $0<m<n$ such that  $m \sim n\sim (n-m),$ the equality
\begin{equation}\label{lim-Hnm}
 \lim_{n\to\infty}
 \|H-H_n^m\| =0,  
\end{equation}
holds by virtue of \eqref{conv-R} applied to  $Hf\in C^0$.\newline
Let us prove that $I-H_n^m:C^0\to C^0$ is invertible and the norm of its inverse is uniformly bounded w.r.t. $n\sim m$. To this aim we consider the following decomposition 
\begin{equation}\label{eq:tmp1}
		I-H_n^m=(I-H) (I-(I-H)^{-1}(H_n^m-H)). 
	\end{equation}
    Since, by assumption, $(I-H)$ is invertible, we focus on the second factor $(I-(I-H)^{-1}(H_n^m-H))$.
By \eqref{lim-Hnm}, given  a positive $\varepsilon<1/\|(I-H)^{-1}\|$, there  exists $\bar{n}=\bar{n}(\varepsilon)\in \mathbb{N}$ such that for all $n > \bar{n}$ we have 
	\begin{equation*}
		\|H-H_n^m\|<\varepsilon \quad \implies \quad \| (I-H)^{-1}(H_n^m-H) \|\leq \|(I-H)^{-1}\|\varepsilon <1.
	\end{equation*}
	 Hence, from Neumann series and contractive mapping arguments (see, e.g. \cite[Theorem A.1]{Atkinson_1997}) the operator $I-(I-H)^{-1}(H_n^m-H)$ is invertible and  we have
     \[
     \|(I-(I-H)^{-1}(H_n^m-H))^{-1}\|\leq \frac 1{1- \| (I-H)^{-1}\|\|(H_n^m-H) \|}\leq \frac 1{1- \varepsilon \| (I-H)^{-1}\|},
     \]
     which proves the existence of the unique solution $f_n^m$ of \eqref{FIE-project}, and that
 \begin{equation}\label{bound-IH}
     \|(I-H_n^m)^{-1}\|\leq \frac{\|(I-H)^{-1}}{1- \varepsilon \| (I-H)^{-1}\|}\leq \C\ne \C(n,m).
 \end{equation}
 Moreover, since $f_n^m=g_n^m+ H_n^m(f_n^m)$, we have $f_n^m\in\R_n^m$.
 \newline
 To establish  \eqref{lim-cond}, we start from the identities 
	 \begin{align*}
    (I-H_n^m)-(I-H)&=  H-H_n^m,\\
	(I-H)^{-1}-(I-H_n^m)^{-1}&=(I-H)^{-1}(H-H_n^m)(I-H_n^m)^{-1},
	 \end{align*}
	 that, by the previous results, imply the following
    \begin{align*}
    \left|\|I-H_n^m\|-\|I-H\|\right|&\leq \| H-H_n^m\|\to 0, \ \mbox{as $n\to \infty,$}\\
	\left|\|(I-H)^{-1}\|-\|(I-H_n^m)^{-1}\|\right| &\leq\|(I-H)^{-1}\|\ \|H-H_n^m\|\ \|(I-H_n^m)^{-1}\| \to 0, \ \mbox{as $n\to \infty$}
	 \end{align*}
     and \eqref{lim-cond} follows.
   \newline  
    Finally, the identities
    \begin{equation*}
        (I-H_n^m)f^*= g+ (H-H_n^m)f^* \qquad \text{and} \qquad
        (I-H_n^m)f_n^m= g_n^m,
    \end{equation*}
    imply
    \[
    f^*-f_n^m=(I-H_n^m)^{-1}\left[(g-g_n^m)+(H-H_n^m)f^*\right]
    \]
    and the estimate \eqref{er-col} follows applying \eqref{conv-R} to the functions $g,\ Hf^*\in C^0$, and taking into account that $ \|(I-H)^{-1}\|$ is uniformly bounded w.r.t. $n$ (cf. \eqref{bound-IH}.\vspace{.2cm}
\end{proof}

Let us assume that \eqref{FIE-project} has a unique solution  $f_n^m\in \R_n^m$.  Recalling \eqref{proj}, it admits the following expansion
\begin{equation}\label{fnm}
     f_n^m(x)=\sum_{i=1}^{n+1}f_i\  r_{n,i}^m(x), \qquad f_i=f_n^m(x_i), \qquad \qquad |x|\leq 1. 
\end{equation}
The unknown coefficients $\mathbf{f}=[f_1,\ldots f_{n+1}]^T$ are the solution of the following collocation system
\begin{equation}\label{sist-col-components}
    (I-H_n^m)f_n^m(x_i)=g_n^m(x_i), \qquad i=1,\ldots, n+1,   
\end{equation}
which, by virtue of \eqref{interp}, yields
\begin{equation}\label{sist-col}
    (\mathcal{I}-\mathcal{A})\mathbf{f}=\mathbf{g}, 
\end{equation}
where $\mathbf{g}=[g(x_1),\ldots g(x_{n+1})]^T$, $\mathcal{I}$ denotes the identity matrix of order $n+1$, and
\begin{equation}\label{A-def}
    \mathcal{A}=[A_{i,j}], \qquad A_{i,j}=\int_{-1}^1 h(x_i,y)r_{n,j}^m(y)w(y)dy, \qquad i,j=1,\ldots,n+1.   
\end{equation}
For any matrix $M=[M_{i,j}]$, let $\kappa_\infty(M):=\left\|M \right\|_\infty \, \left\|M^{-1}\right\|_\infty$ denote its condition number with respect to  the  norm 
\[
    \|M\|_\infty:=\sup_{\bm{v}\ne \bm{0}}\frac{\|M\bm{v}\|_\infty}{\|\bm{v}\|_\infty}=\max_i\sum_j |M_{i,j}|, \qquad \|\bm{v}\|_\infty:=\max_i|v_i|.
\]
The following result states that the linear system \eqref{sist-col} is well conditioned if the original equation \eqref{FIE-op} is well conditioned.
\begin{theorem}\label{th-cond-col}
Under the assumptions of Theorem \ref{th-collo}, let the discrete operator $I-H_n^m: \R_n^m\to\R_n^m$ be represented by the matrix  $\mathcal{I}-\mathcal{A}$ with respect to the basis \eqref{rnm}. Then there exists a constant $C>0$ such that we have
\begin{equation}\label{eq-cond-col}
    \C^{-1}\ \kappa(I-H)\leq \liminf_{\substack{n\to \infty \\ m \sim n}}\kappa_\infty(\mathcal{I}-\mathcal{A})\leq \limsup_{\substack{n\to \infty \\ m \sim n}}\kappa_\infty(\mathcal{I}-\mathcal{A})\leq \C\ \kappa(I-H).
\end{equation}
\end{theorem}
\begin{proof}
Given two vectors $\bm{v}=[v_1,\dots,v_{n+1}]^\top\in \mathbb{R}^{n+1}$ and $\bm{u}=[u_1,\dots,u_{n+1}]^\top\in \mathbb{R}^{n+1},$ we define the following functions
\begin{equation*}
     V(x)=\sum_{i=1}^{n+1}v_i r_{n,i}^m(x), \qquad \qquad U(x)=\sum_{i=1}^{n+1}u_i r_{n,i}^m(x), \qquad \qquad |x|\leq1,
\end{equation*}
which belong to the rational space $\R_n^m.$ By construction, $V =0$ iff $\bm{v}=\bm{0}$ and $U =0$ iff $\bm{u}=\bm{0}.$
Since the matrix $\mathcal{I}-\mathcal{A}$ represents the discrete operator $I-H_n^m: \R_n^m\to\R_n^m$ with respect to the basis of \eqref{rnm}, the following characterization 
\begin{equation*}
    U=(I-H_n^m)V \qquad \iff \qquad \bm{u}=(\mathcal{I}-\mathcal{A}) \ \bm{v},
\end{equation*}
holds true. Furthermore, from \eqref{bound-R}, we have $\|V\|_\infty\sim \|\bm{v}\|_\infty$ and $\|U\|_\infty\sim \|\bm{u}\|_\infty,$ that in turn imply 
\begin{equation*}
    \begin{split}
        &\|\mathcal{I}-\mathcal{A}\|_\infty=\sup_{\bm{v}\neq \bm{0}}\dfrac{\|(\mathcal{I}-\mathcal{A})\bm{v}\|_\infty}{\|\bm{v}\|_\infty}\sim \sup_{V\neq0} \dfrac{\|(I-H_n^m)V\|_\infty}{\|V\|_\infty},\\
        &\|(\mathcal{I}-\mathcal{A})^{-1}\|_\infty=\sup_{\bm{u}\neq \bm{0}}\dfrac{\|(\mathcal{I}-\mathcal{A})^{-1}\bm{u}\|_\infty}{\|\bm{u}\|_\infty}\sim \sup_{U\neq0} \dfrac{\|(I-H_n^m)^{-1}U\|_\infty}{\|U\|_\infty}.
    \end{split}
\end{equation*}
Therefore, 
\begin{equation*}
    \kappa_\infty(\mathcal{I}-\mathcal{A})=\|\mathcal{I}-\mathcal{A}\|_\infty \ \|(\mathcal{I}-\mathcal{A})^{-1}\|_\infty\sim \|I-H_n^m\| \ \|(I-H_n^m)^{-1}\| =\kappa (I-H_n^m),
\end{equation*}
so that by  \eqref{lim-cond} 
we get
\begin{gather*}
    \liminf_{\substack{n\to \infty \\ m \sim n}}\kappa_\infty(\mathcal{I}-\mathcal{A})\sim \liminf_{\substack{n\to \infty \\ m \sim n}}\kappa (I-H_n^m)=\kappa (I-H),\\
    \limsup_{\substack{n\to \infty \\ m \sim n}}\kappa_\infty(\mathcal{I}-\mathcal{A})\sim \limsup_{\substack{n\to \infty \\ m \sim n}}\kappa (I-H_n^m)=\kappa (I-H),
\end{gather*}
which yield \eqref{eq-cond-col}. 
\end{proof}

In conclusion, following a classical approach applied to the projection $R_n^m$, the RC method approximate the unique solution of \eqref{FIE} at the desired precision by means of the rational function $f_n^m$ in \eqref{fnm}. The coefficients $f_j$ are computed by solving the linear system \eqref{sist-col}, which is  well--conditioned if the original problem is well--conditioned. Moreover, according to the error estimates \eqref{er-col}, the convergence rate depends on the regularity of $g$ and $Hf^*$. 
\newline
The main effort in this procedure is the computation of the integrals in \eqref{A-def}. In this regard, we observe that if we approximate them by the Gaussian rule \eqref{Gauss}, due to \eqref{interp}, we get exactly the linear system associated to the classical Nystr\"om method, namely
\begin{equation}\label{Nystrom-sist}
    \bar f_i-\sum_{j=1}^{n+1} \lambda_j h(x_i, x_j)\bar f_j=g(x_j), \qquad j=1,\ldots, n+1.
\end{equation}
Hovewer, in such case, instead of computing the Nystr\"om interpolant
\begin{equation}\label{Nystrom-inter}
     \bar f(x)=g(x)+\sum_{i=1}^{n+1}\lambda_i \bar f_i \ h(x, x_i),\qquad |x|\leq 1,  
\end{equation}
we approximate the solution by the rational function
\begin{equation}\label{fnm-MN}
    \bar f_n^m(x)=\sum_{i=1}^{n+1}\bar f_i r_{n,i}^m(x), \qquad \qquad |x|\leq 1.  
\end{equation}
We refer to the scheme \eqref{Nystrom-sist}, \eqref{fnm-MN} as the Modified Nystr\"om (MN) method. Such a method has the advantage of using only some sampling of the known functions $g$ and $h$, but its performance turns out to be almost equivalent to Nystr\"om method, as shown in the numerical experiments of Section \ref{sec:Numerical_Tests}. 

A different approach to approximate the integrals in \eqref{A-def} is proposed in the following section.

\section{The RDC method}\label{sec:RDC_Method}
Here, we develop a Rational Discrete Collocation (RDC) method that employs the rational projector $R_n^m$ and  overcomes the  problem of computing the integrals in \eqref{A-def}. We first note that, by \eqref{rnm}, such integrals can be written as 
\begin{equation}\label{Aij}
    A_{i,j}=\lambda_j \int_{-1}^1 [\lambda_m(y)h(x_i,y)] K_n(x_j,y)K_m(x_j,y)w(y)dy, \qquad i,j=1,\ldots,n+1,
\end{equation}
where 
\begin{equation}\label{lambda}
    \lambda_m(y)=\frac 1{K_m(y,y)}, \qquad |y|\leq 1,
\end{equation}
denotes the Christoffel function of order $m$ corresponding to the Jacobi weight $w$. We then observe that, by replacing the functions
\begin{equation}\label{Fi}
    F_i(y):= \lambda_m(y)h(x_i,y), \qquad i=1,\ldots, n+1,  
\end{equation} 
in \eqref{Aij} with suitable polynomials $\mathcal{P}_i(y)$ of degree at most $n-m$, the resulting integrals can be exactly evaluated as follows
\begin{equation}\label{inva-Pi}
    \int_{-1}^1 \mathcal{P}_i(y) K_n(x_j,x)K_m(x_j,x)w(x)dx=\mathcal{P}_i(x_j),\qquad i,j=1,\ldots, n+1, \qquad \forall \mathcal{P}_i\in\PP_{n-m},   
\end{equation}
where we used the reproducing property \eqref{inva} with $P(y)=\mathcal{P}_i(y)K_m(x_j, y)\in \PP_{n}$.
 
This consideration constitutes the underlying idea of the RDC method presented in the following subsections.
\subsection{The choice of the polynomial approximation}\label{subsec:choice_Polynomials}
For each arbitrarily fixed $x\in[-1,1]$, we propose to approximate the function
\begin{equation}\label{F}
    F(y)=\lambda_m(y) h(x,y), \qquad |y|\leq 1,
\end{equation}
by the filtered de la Vall\'ee Poussin (VP) type interpolation polynomial $V_N^MF(y),$ defined as follows (see, e.g., \cite{Th-1999, Woula_Uniform})
\begin{equation}\label{VP}
    V_N^MF(y)=\sum_{s=1}^N F(y_s)\Phi_{N,s}^M(y),\qquad |y|\leq 1,   
\end{equation}
where 
\begin{equation}\label{ys}
   y_s=\cos\left(\dfrac{(2s-1)\pi}{2N}\right),\qquad s=1,\ldots,N, 
\end{equation}
are the first kind Chebyshev zeros of order $N.$ In \eqref{VP}, the basis functions $\Phi_{N,s}^M$ are the fundamental VP-Chebyshev polynomials of parameter $N$ and $M$, given by \cite{Woula_Uniform}
\begin{equation}\label{fund-VP}
    \Phi_{N,s}^M(y)=\frac \pi N\sum_{r=0}^{N+M-1}
    \mu_{r}\tilde T_r(y_s)\tilde T_r(y),\qquad s=1,\ldots,N,\qquad |y|\leq 1,
\end{equation}
with
\begin{equation*}
    \tilde T_r(y)=\displaystyle\sqrt{\frac 2\pi}
    \begin{cases}
        \displaystyle \sqrt{1/2} & \mbox{if}\quad r=0,\\
         \displaystyle \cos[r(\arccos y)] & \mbox{otherwise},
    \end{cases}
    \qquad
    \mu_r=\begin{cases}
        1 & \mbox{if}\quad 0\leq r\leq N-M,\\
        \dfrac{M+N-r}{2M} & \mbox{if}\quad
    N-M< r< N+M.
    \end{cases}
\end{equation*}
The free VP parameters $0<M<N$ are required to satisfy the following conditions
\begin{align}
    \label{cond-1}
   & N+M-1\leq n-m,\\
   \label{cond-2}
   & M\sim N\sim (N-M).
\end{align}
The former ensures that $V_N^MF\in\PP_{n-m}$, allowing us to apply \eqref{inva-Pi} with $\mathcal{P}_i=V_N^MF_i$. The latter guarantees that $V_N^MF$ is a stable near--best approximation of $F$.

More precisely, for the purposes of this study, we recall the following properties  (see, e.g. \cite{Woula_Donatella_Cheb_1D, OT-DRNA21, Woula_Uniform, Woula_Weighted})
\begin{enumerate}
\item{\bf Interpolation.} For all $0<M<N$ we have
\begin{equation}\label{interp_V}
    V_N^MF(y_s)=F(y_s), \qquad s=1,\ldots,N.
\end{equation}
\item{\bf Stability.} For all $0<M<N$ such that $M\sim N$,  we have
\begin{align}
\label{equiVP-inf}
    \|V_N^M F\|_\infty &\sim \max_s |F(y_s)|,\\
\label{equiVP-1}
    \|V_N^M F\|_{L^1_v}&\sim  \frac 1 N\sum_{s=1}^N |F(y_s)|,
\end{align}
where $L^1_v=\{F \ : \ \|F\|_{L^1_v}<\infty\}$, with
\[
    \|F\|_{L^1_v}= \int_{-1}^1|F(y)|v(y)dy, \qquad \text{with} \qquad v(y)= \frac{1}{\sqrt{1-y^2}}.
\]
\item{\bf Convergence.} If \eqref{cond-2} holds for sufficiently large $N\in\NN$ then, for all bounded and Riemann integrable function $F\in L^1_v$, we have 
\begin{equation}\label{lim-VP}
    \lim_{N\to\infty}\|V_N^M F-F\|_{L^1_v}=0,
\end{equation}
where the rate of convergence depends on the smoothness properties of the target function $F$. In particular, 
if $F\in BV$ is a function of bounded variation on $[-1,1]$, then
\[
    \|V_N^M F-F\|_{L^1_v}\leq \frac \C N \int_{-1}^1 |dF(y)|, \qquad \C\ne \C(N,M,F).
\]
Furthermore, if $F\in L^1_v$ is locally continuous on $[-1,1]$, then
\[
    \|V_N^M F-F\|_{L^1_v}\leq \frac \C N \int_0^\frac 1n \frac{\Omega_\varphi^r(F,t)_{L^1_v}}{t^2} dt, \qquad r>1,\qquad \C\ne \C(N,M,F).
\]
\item{\bf Error estimates}
The previous properties hold for arbitrary functions $F(y)$. In the case of our interest we prove the following result which will be used later on.
\begin{theorem}\label{prop-MarciVP}
Let the positive integers $N>M$ be such that \eqref{cond-2} holds, and let $\lambda_m$ be the Christoffel function \eqref{lambda} corresponding to a Jacobi weight $w$ and  $m\in\NN$ such that $m\sim N$. For all $f\in L^1_w$, if $F(y)=f(y)\lambda_m(y)$ then 
\begin{equation}\label{MarciVP}
\int_{-1}^1|V_N^MF(y)|k_m(y,y)w(y)dy\leq \C\sum_{s=1}^N |f(y_s)|\lambda_m(y_s), \qquad \C\ne\C(N,m,f).
\end{equation} 
Moreover, set $\varphi(y)=\sqrt{1-y^2}$, if $fw\varphi$ is bounded and $fw$ is Riemann integrable on $[-1,1]$, then we have
\begin{equation}\label{lim-VP1}
    \lim_{N\to\infty}\int_{-1}^1|f(y)-V_N^MF(y)k_m(y,y)|\ w(y)dy=0, \qquad F(y)=f(y)\lambda_m(y).
\end{equation} 
\end{theorem}
\begin{proof}
    Let us first prove \eqref{MarciVP}. From the definition \eqref{VP}, we have 
    \begin{align}
        \int_{-1}^1|V_N^MF(y)|k_m(y,y)w(y)dy&=\int_{-1}^1 \left|\sum_{s=1}^N f(y_s)\lambda_m(y_s)\phi_{N,s}^M(y)\right| k_m(y,y)w(y)dy\notag\\
        &\leq  \sum_{s=1}^N |f(y_s)|\lambda_m(y_s)\int_{-1}^1 \left|\phi_{N,s}^M(y)\right| k_m(y,y)w(y)dy.
        \label{eq1}
    \end{align}
    Regarding the last integral, since $\phi_{N,s}^M(y)k_m(y,y)$ is a polynomial, we apply the Remez type inequality \cite[(8.1.4)]{DT}: for all $\nu\in\NN$ and for every polynomial $P\in\PP_{\nu}$, 
    \begin{equation}\label{Remez}
        \int_{-1}^1 |P(y)|w(y)\,dy\le \C\int_{-1+\C \nu^{-2}}^{\,1-\C \nu^{-2}}|P(y)|w(y)\,dy,
        \qquad \C\ne\C(\nu,P).
    \end{equation}
    Since in our regime $m\sim N\sim M$ the involved polynomials have degree $\C\,m$,  \eqref{Remez} allows us to restrict the integration to
    \[
    I_m:=[-1+\C m^{-2}, 1-\C m^{-2}],
    \]
    and on this interval, by \eqref{CF}, we have
    \begin{equation}\label{stima-k}
     k_m(y,y)\sim \frac m{w(y)\varphi(y)}, \qquad \forall y\in I_m,\qquad \varphi(y):=\sqrt{1-y^2}.
    \end{equation} 
    We specify that here and in the following the involved constants $\C>0$  are always supposed independent of the parameters of our interest, namely $\C\ne \C(N,M,m,f)$.\newline 
    By \eqref{Remez}, \eqref{stima-k}, \eqref{equiVP-1} and $m\sim N$,
    we get
    \begin{align*}
        \int_{-1}^1 \left|\phi_{N,s}^M(y)\right| k_m(y,y)w(y)dy&\leq \C\int_{I_m} \left|\phi_{N,s}^M(y)\right| k_m(y,y)w(y)dy 
        \\
        &\leq \C\ m \int_{I_m}\left|\phi_{N,s}^M(y)\right| \frac{dy}{\varphi(y)}
        =  \C\ m \int_{I_m}\left|V_N^M \phi_{N,s}^M(y)\right| \frac{dy}{\varphi(y)} \\
        &\leq \C\frac{m}{N} \sum_{r=0}^N \left|\phi_{N,s}^M(y_r)\right|= \C\frac{m}{N}\leq \C 
    \end{align*}
    and \eqref{MarciVP} follows from \eqref{eq1}.

    To prove \eqref{lim-VP1}, we recall the following estimate of the Christoffel function (see, e.g. \cite[(18)]{NevaiVert})
    \begin{equation}\label{stima-Nevai}
        \lambda_m(y)\sim \begin{cases}
            \dfrac{w(y)\varphi(y)}{m}, \qquad \ &  y\in I_1=I_m:=[-1+\C m^{-2},\ 1-\C m^{-2}],\\[2pt]
            \dfrac{w(1-\C m^{-2})}{m^2},  &  y\in I_2=[1-\C m^{-2},\ 1],\\[2pt]
            \dfrac{w(-1+\C m^{-2})}{m^2},  &  y\in I_3=[-1, \ -1+\C m^{-2}],
        \end{cases}
        \quad  \ \C\ne\C(m,y).
    \end{equation}
    We also recall that the Christoffel function satisfies (see \cite[Ch.~4]{Nevai} and \cite{MateNevaiTotik_1991}):
    \begin{equation}\label{MNT}
        \lim_{m\to\infty} m\,\lambda_m(y)=\pi\,\varphi(y)\,w(y)\qquad\text{uniformly on compact subsets of }(-1,1).
    \end{equation}
Moreover, set
$
E_n(f)_{L^1_w}=\inf_{P\in\PP_n}\|(f-P)w\|_1
$, 
the Weierstrass  approximation theorem yields
\begin{equation}\label{best-L1}
  \lim_{n\to\infty}E_{n}(f)_{L^1_w}=0, \qquad \forall f\in L^1_w . 
\end{equation}
    Set  $G=mF=m\lambda_m f$ and let $P^*\in\PP_{n}$ be the polynomial of best approximation of $f$ in $L^1_w$, i.e. $\|(f-P^*)w\|_{1}=E_{n}(f)_{L^1_w}$. By adding and subtracting $P^*$ and applying the Remez inequality \eqref{Remez} (all the polynomials involved having degree $\C\,m$ in the regime $m\sim n\sim N\sim M$),
    we get
    \begin{align}
        &\int_{-1}^1\big|f(y)-V_N^MF(y)\,k_m(y,y)\big| w(y)\,dy\notag\\
        &\le \int_{-1}^1|f(y)-P^*(y)|w(y)dy
           +\int_{-1}^1\big|P^*(y)-V_N^MF(y)\,k_m(y,y)\big| w(y)\,dy\notag\\
        &\le E_{n}(f)_{L^1_w}+\C\int_{I_m}\big|P^*(y)-V_N^MF(y)\,k_m(y,y)\big| w(y)\,dy\notag\\
        &\le E_{n}(f)_{L^1_w}+\C\int_{I_m}|P^*(y)-f(y)|w(y)dy
           +\C\int_{I_m}\big|f(y)-V_N^MF(y)\,k_m(y,y)\big| w(y)\,dy\notag\\
        &\le \C\,E_{n}(f)_{L^1_w}
           +\C\int_{I_m}\big|G(y)-V_N^MG(y)\big|\,\frac{dy}{\varphi(y)},\label{eq:reduce-Im}
    \end{align}
    where in the last line we used $f=G\,k_m/m$,  and  \eqref{stima-Nevai} that yields $k_m(y,y)w(y)\sim m/\varphi(y)$ on $I_m$. 
 \newline   
    Thus, by \eqref{best-L1}, we are reduced to prove that \begin{equation}\label{tesi-G}
      \lim_{N\to\infty} 
      \int_{I_m}|G(y)-V_N^MG(y)| v(y)dy= 0, \qquad G(y)=m\lambda_m(y)f(y).
    \end{equation}
    Consider the functions: 
    \[
\psi(y):=\pi\,\varphi(y)\,w(y)\,f(y),\qquad |y|\le 1,\qquad\mbox{and}\qquad 
\sigma_m(y):=\frac{m\lambda_m(y)}{\pi\varphi(y) w(y)}, \quad |y|<1.
    \]
Note that the former is {bounded} (since $fw\varphi$ is bounded) and {Riemann integrable} (since $fw$ is Riemann integrable and $\varphi$ is continuous); the latter, by \eqref{stima-Nevai} and \eqref{MNT},  satisfies
\begin{equation}\label{sig-prop-1}
\C^{-1}\le\sigma_m(y)\le\C,\qquad \forall y\in I_m \qquad   \C\ne\C(m,y),
\end{equation}
and 
\begin{equation}\label{sig-prop-2}
\lim_{m\to\infty}\sigma_m(y)=1\qquad \mbox{uniformly on the compact subsets of $(-1,1)$}.    
\end{equation}
Writing $G=\psi+(\sigma_m-1)\psi$, we obtain
\begin{align*}
    \int_{I_m}|G(y)-V_N^MG(y)|v(y){dy}
    &\le \|\psi-V_N^M\psi\|_{L^1_v}
     +\int_{I_m}|(\sigma_m(y)-1)\psi(y)|\,v(y){dy}+\\
    &+\big\|V_N^M[(\sigma_m-1)\psi]\big\|_{L^1_v}=:A+B+C
\end{align*}
    where in $A,C$ we enlarged the integration domain from $I_m$ to $[-1,1]$.

    \emph{Term $(A)$.} Since $\psi$ is a \emph{fixed} bounded and Riemann integrable function, the convergence property \eqref{lim-VP} applies directly and yields \[
    A=\|\psi-V_N^M\psi\|_{L^1_v}\to 0  \qquad\mbox{as $N\to\infty$}.
\]
    \emph{Term $(B)$.} Note that we cannot apply \eqref{lim-VP} to $(\sigma_m-1)\psi$, since this function depends on $m$ (hence on $N$); we therefore estimate this term directly. For arbitrarily fixed $0<\delta<1$, split \[
    I_m=I^\delta\cup(I_m\setminus I^\delta)\qquad \mbox{with}\qquad I^\delta:=[-1+\delta,\ 1-\delta].
    \]
    On $I^\delta$, since $\sigma_m(y)\to1$ uniformly on compact subsets (cf. \eqref{sig-prop-2}),
\[
\int_{I^\delta}|(\sigma_m(y)-1)\psi(y)|\,v(y){dy}
\le\Big(\sup_{y\in I^\delta}|\sigma_m(y)-1|\Big)\|\psi\|_{L^1_v}
\xrightarrow[N\to\infty]{}0 .
\]
    On $I_m\setminus I^\delta\subseteq\{|y|\ge1-\delta\}$, using $|\sigma_m-1|\le\C$ (cf. \eqref{sig-prop-1}),
    \[
       \int_{I_m\setminus I^\delta}|(\sigma_m(y)-1)\psi(y)|\,v(y){dy}
        \le\C\int_{\{|y|\ge1-\delta\}}|\psi(y)|\,v(y){dy}=:\C\,\omega(\delta), \qquad \C\ne\C(\delta).
    \]
Therefore,
    \begin{equation}\label{limsup}
\limsup_{N\to\infty}B\le\C\,\omega(\delta), \qquad \forall  \delta\in (0,1).      
    \end{equation} 
 On the other hand, since $\psi\in L^1_v$, the absolute continuity of the Lebesgue integral gives 
    \begin{equation}\label{omega}
    \omega(\delta):=\int_{\{|y|\ge1-\delta\}}|\psi(y)|\,v(y)\,dy  \to0\qquad \mbox{as}\quad \delta\to0.  
    \end{equation}
Thus, since $B\ge 0$ does not depend on $\delta$, by \eqref{limsup}-\eqref{omega}, we get $0\le \liminf_{N\to\infty}B\le \limsup_{N\to\infty}B\le 0$, i.e. $B\to0$.

    \emph{Term $(C)$.} For arbitrarily fixed $\delta\in (0,1)$, set
    \[
    J_1=\{s :\ y_s\in I^\delta\}, \qquad J_2=\{s :\ y_s\in I_m\setminus I^\delta\}, \qquad J_3=\{s :\ y_s\in[-1,1]\setminus I_m\}.
    \]
    By \eqref{equiVP-1},  \begin{equation}\label{eq:C-start}
    C= \big\|V_N^M[(\sigma_m-1)\psi]\big\|_{L^1_v}\le \frac\C N\left\{ \sum_{s\in J_1}+\sum_{s\in J_2}+\sum_{s\in J_3}\right\}
\big|\sigma_m(y_s)-1\big|\,\big|\psi(y_s)\big|=:S_1+S_2+S_3.
    \end{equation}

    To prove that these sums tend to zero as $N\sim m\to\infty$, the arguments are the same as for the  term B, 
    the only new point being that here we deal with Gauss--Chebyshev quadrature sums. In this regard, we recall that for any bounded and Riemann integrable function $\zeta\in L^1_v$,
    \begin{equation}\label{Gauss-conv}
  \lim_{N\to\infty}\left[\frac\pi N\sum_{s=1}^N\zeta(y_s)\right]=\int_{-1}^1\zeta(y)v(y)dy,     
    \end{equation}   
     and therefore   \begin{equation}\label{quad-sum}
    \frac 1N\sum_{s=1}^N|\zeta(y_s)|=\bigO(1), \qquad \mbox{for all bounded and Riemann integrable $\zeta$}
    \end{equation}
   \emph{(i) Sum $S_1$ }
   Since $|\psi|$ is bounded and Riemann integrable, by \eqref{quad-sum} and \eqref{MNT}
    \begin{align*}
        S_1&= \frac{\C}{N}\sum_{s\in J_1}|\sigma_m(y_s)-1|\,|\psi(y_s)|
        \le \sup_{y\in I^\delta}|\sigma_m(y)-1|\left[\frac \C N\sum_{s=1}^N|\psi(y_s)|\right]\\
        &\le \C \sup_{y\in I^\delta}|\sigma_m(y)-1|\xrightarrow[N\to\infty]{}0.
    \end{align*}
    \emph{(ii) Sum  $S_2$.} Using \eqref{sig-prop-1} and \eqref{quad-sum} with  $\zeta=\psi \cdot\chi_{[-1,1]\setminus I^\delta}$, 
    \begin{align*}
         S_2&=\frac{\C}{N}\sum_{s\in J_2}|\sigma_m(y_s)-1|\,|\psi(y_s)|
        \le   \frac{\C}{N}\sum_{y_s\in I_m\setminus I^\delta}|\psi(y_s)|\\
        &\le \C\cdot\frac{\pi}{N}\!\!\sum_{y_s\in\{|y|\ge1-\delta\}}\!\!|\psi(y_s)|
        \xrightarrow[N\to\infty]{}\C\int_{\{|y|\ge1-\delta\}}|\psi(y)|\,v(y)\,dy=\C\,\omega(\delta),
    \end{align*}
 with $\omega(\delta)$ the same quantity defined in \eqref{omega}.
Hence, reasoning as for the term B, we get $S_2\to 0$.

    \emph{(iii) Sum $S_3$.}
    Note that 
    the boundary nodes $y_s\in [-1,1]\setminus I_m$ are only $\bigO(1)$ in number. Indeed, from $1-y_s\le\C m^{-2}$ and $y_s=\cos t_s$ with $t_s:=(2s-1)\pi/(2N)\sim \sin(t_s/2)$, one gets $\sin^2(t_s/2)\le \C m^{-2}$ that gives $2s-1
    \le \C N/m=\bigO(1)$ (and symmetrically near $-1$). 
    \newline
    Also note that $\sigma_m \psi=G=m\lambda_mf$ and that, by \eqref{stima-Nevai}, $\lambda_m(y_s)\le\C[m^{-2\alpha-2}+\C m^{-2\beta-2}]$ holds,  for $w=v^{\alpha,\beta}$ and every $s\in J_3$. 
    Consequently
    \begin{align*}
        S_3&=\frac\C N\sum_{s\in J_3}|(\sigma_m\psi)(y_s)-\psi(y_s)|\le \frac\C N \sum_{s\in J_3}|G(y_s)|+\frac\C N \sum_{s\in J_3}|\psi(y_s)|\\
        &\le \C \frac mN \sum_{s\in J_3}(m^{-2\alpha-2}+ m^{-2\beta-2}) |f(y_s)|+\frac \C N \sum_{s\in J_3}\|\psi\|_\infty\\
        &\le \C \left(\max_{s\in J_3}|f(y_s)|\right)[m^{-2(\alpha+1)}+m^{-2(\beta+1)}]+\frac \C N
        \xrightarrow[N\to\infty]{}0.
    \end{align*}
In conclusion, by \emph{(i)--(iii)} we get $C\to 0$ as $N\to\infty$, and 
    collecting $(A)$--$(C)$ we obtain \eqref{tesi-G}, which together with \eqref{eq:reduce-Im} yields \eqref{lim-VP1}.
\end{proof}

\begin{remark}\label{rem-VP}
    We remark that, according with the previous proof, the rate of convergence in \eqref{lim-VP1} depends on the smoothness properties  of the function 
    $G(y)=m\lambda_m(y)f(y)$ in $L^1_v$. 
    In this regard, 
    we recall that \cite[(23)]{NevaiVert}
        \begin{equation}\label{stima-Nevai_2}
            \lambda_m^\prime(y)\sim 
            \begin{cases}
                \dfrac{w(y)}{m\varphi(y)},  &  y\in I_1=[-1+\C m^{-2},\ 1-\C m^{-2}], \\[2pt]
                w(1-\C m^{-2}),  &  y\in I_2=[1-\C m^{-2},\ 1],\\[2pt]
                w(-1+\C m^{-2}), &  y\in I_3=[-1, \ -1+\C m^{-2}],
            \end{cases}
            \quad \text{where} \ \varphi(y)=\sqrt{1-y^2}.
        \end{equation}
\end{remark}

 \item{\bf Fast computation.}\label{P5}
    In conclusion, we focus on an important aspect for  applications: how to compute $V_N^MF$ efficiently. To this aim, we recall some alternative forms of this polynomial. 

Indeed, $V_N^MF$ represents a projection of  $F$ onto  the polynomial space
\begin{equation}\label{VP-basis}
  \V_N^M=\operatorname{span}\{\Phi_{N;s}^M\:\ s=1,\ldots, N\},  
\end{equation}
and formula \eqref{VP} yields such projection in term of the interpolating basis \eqref{fund-VP}, more explicitly
\begin{equation}\label{VP-cr}
    V_N^MF(y)={\frac\pi N}\sum_{r=0}^{N+M-1}c_r\mu_r \tilde T_r(y), \qquad c_r=\sum_{s=1}^NF(y_s)\tilde T_r(y_s).
\end{equation}

However, other bases of $\V_N^M$ are known in the literature, in particular orthogonal bases with respect to the scalar product
\[
\langle f, g \rangle_{L^2_v}=\int_{-1}^1f(y)g(y)v(y)dy, \qquad v(y)=\frac 1{\sqrt{1-y^2}}.
\]
In fact, it is known that \cite[Thm.4.3]{Woula_Uniform}
\begin{equation}\label{basis-q}
    \V_N^M=\operatorname{span}\{q_{r,N}^M\:\ r=0,\ldots, N-1\},  
\end{equation}
where 
\begin{equation}\label{q}
q_{r,N}^M(y)=\left\{\begin{array}{lr}
  \tilde T_r(y),   &  0\leq r\leq N-M,\\
  \mu_r \tilde T_r(y)-\mu_{2N-r}\tilde T_{2N-r}(y),   & 
 \qquad  N-M<r<N+M,
\end{array}\right.   
\qquad r=0,\ldots, N-1,
\end{equation}
with
\[
\langle q_{r,N}^M,\ q_{s,N}^M \rangle_{L^2_v}=\delta_{r,s} \nu_r,\qquad \nu_r:=
\left\{\begin{array}{lr}
  1   &  0\leq r\leq N-M,\\
  \frac{M^2+(N-r)^2}{2M^2}  & N-M<r<N+M.
\end{array}\right.
\]
Moreover, in the recent paper  \cite{TB-ORTOwave}, 
it has been proved that
\begin{equation}\label{basis-OVP}
    \V_N^M=\operatorname{span}\{\tilde \varphi_{N,s}^M\:\ s=1,\ldots, N\},  
\end{equation}
where the polynomials
\[
    \tilde\varphi_{N,s}^M(y)=\sqrt{\frac\pi N}\sum_{r=0}^{N-1}\frac{p_r(y_s)}{\sqrt{\nu_r}}q_{r,N}^M(y),\qquad s=1,\ldots, N,
\]
form an orthonormal and localized polynomial basis satisfying [...]
\begin{equation*}
    \langle\tilde\varphi_{N,s}^M, \ \tilde\varphi_{N,r}^M\rangle_{L^2_v}=\delta_{r,s} \qquad \text{and} \qquad |\tilde\varphi_{N,s}^M(y)|\leq\C \frac{\sqrt{N}}{1+NM|y-y_s|^2}, \qquad |y|\leq 1,
\end{equation*}
with $\C\ne \C(y,s, N,M)$.

Consequently,
in terms of the basis \eqref{basis-q}, we have
\begin{equation}\label{VP-q}
    V_N^MF(y)=\sqrt{\frac\pi N}\sum_{r=0}^{N-1}c_r q_{r,N}^M(y), \qquad c_r=\sqrt{\frac \pi N}\sum_{s=1}^NF(y_s)\tilde T_r(y_s),
\end{equation}
whereas, in terms of the basis \eqref{basis-OVP}, we have
\begin{equation}\label{VP-OVP}
     V_N^MF(y)=\sqrt{\frac\pi N}\sum_{s=1}^{N}\tilde c_s \tilde\varphi_{s,N}^M(y), \qquad \tilde c_s=\sqrt{\frac \pi N}\sum_{k=1}^NF(y_k)\tilde \varphi_{N,s}^M(y_k).
\end{equation}

Comparing the previous formulas of $V_N^MF$, we note that the expansion \eqref{VP-q}  is based on less addenda than \eqref{VP-cr} and is simpler than \eqref{VP-OVP}. Moreover,  the following discrete cosine transforms (in Matlab known as DCT of type 2 and 3)
\[
b_r=\sum_{s=1}^N a_s\tilde T_{r}(y_s), \quad r=0,\ldots,N-1,\qquad \qquad a_s=\sum_{r=0}^{N-1} b_r\tilde T_{r}(y_s), \quad s=1,\ldots,N,
\]
can be applied for an efficient computation.

Finally, we remark that the bases \eqref{VP-basis} and \eqref{VP-OVP} have recently been identified as scaling functions, and the corresponding interpolating and orthonormal wavelets have been derived, together with efficient DCT-based decomposition and reconstruction algorithms \cite{Cotronei2025, FT-wave, TB-ORTOwave}. Therefore, the representations \eqref{VP} and \eqref{VP-OVP} naturally provide the foundation for a wavelet multiresolution analysis, enabling sparse representations of the underlying functions. 
\end{enumerate} 
 
\subsection{The approximate equation}
Instead of \eqref{FIE-project}, the approximate equation considered by RDC method is the following
\begin{equation}\label{FIE-rdc}
    (I-\tilde H_n^m)f(x)=g_n^m(x), \qquad |x|\leq 1,  
\end{equation}
where, as usual, $f$ denotes the unknow function of the equation, $g_n^m$ is always given by \eqref{gnm}, with the integers $0<m<n$ such that $m\sim n$, and the operator $\tilde H_n^m$ is defined by
\begin{equation}\label{tilde-Hnm}
    \tilde     H_n^mf(x)=\sum_{i=1}^{n+1}\left[\int_{-1}^1 V_N^M F_{i}(y) k_m(y,y)f(y)w(y)dy\right]r_{n,i}^m(x), \qquad |x|\leq 1,
\end{equation}
where $F_i(y):=h(x_i,y)\lambda_m(y)$, and the positive integers $M<N$ satisfy \eqref{cond-1}-- \eqref{cond-2}.
\begin{proposition}\label{prop-Htilde}
 Under the previous setting and assumptions, let \eqref{hp1-h}--\eqref{hp2-h} hold and, for all $|x|\le 1$, let $h(x,y)w(y)\varphi(y)$ be a bounded function of the variable $y\in[-1,1]$. Then the map $\tilde H_n^m:C^0\to C^0$ is uniformly bounded w.r.t. $n$, i.e., we have
 \[
     \|\tilde H_n^m f\|_\infty\leq \C \|f\|_\infty, \qquad \forall f\in C^0,\qquad\qquad \C\ne \C(n,N,f).
 \]
\end{proposition}
\begin{proof}
In what follows, we always assume that $\C$ indicates a positive constant independent of the parameters in which we are interested, namely $\C\ne \C(n,m,N,M,f,x_i)$.\newline
    Using Prop. \ref{prop-LC} and \eqref{MarciVP}, for all $f\in C^0$, we get
    \begin{align*}
         \|\tilde H_n^mf\|_\infty&\leq \C  \max_{1\leq i\leq n+1}
        \left[\int_{-1}^1 |V_N^M F_{i}(y)|\  |f(y)| \ k_m(y,y) w(y)dy\right] \\
        &\leq    \C \|f\|_\infty \max_{1\leq i\leq n+1}
        \left[\int_{-1}^1 |V_N^M F_{i}(y)|\  k_m(y,y) w(y)dy\right]\\
        &\leq    \C \|f\|_\infty \max_{1\leq i\leq n+1}\left[\sum_{s=1}^N \lambda_m(y_s)|h(x_i,y_s)|\right],
    \end{align*}
    and the statement follows once we show that
    \begin{equation}\label{tesi-prop43}
 \max_{1\leq i\leq n+1}\left[\sum_{s=1}^N \lambda_m(y_s)|h(x_i,y_s)|\right]=\bigO(1).       
    \end{equation}
    To this end we note that, by \eqref{stima-Nevai}, for the interior nodes $y_s\in I_1$ we have $\lambda_m(y_s)\le \C\,w(y_s)\varphi(y_s)/m$, whereas for boundary nodes we get $\lambda_m(y_s)\le \C m^{-2\alpha-2}$ if $y_s\in I_2$, and $\lambda_m(y_s)\le \C m^{-2\beta-2}$ if $y_s\in I_3$. Hence, taking into account that, as already observed in the proof of Thm. \ref{prop-MarciVP} (see {\it (iii)}), the number of boundary nodes is $\bigO(1)$, we get 
    \begin{align*}
        \sum_{s=1}^N \lambda_m(y_s)|h(x_i,y_s)|&\le  \frac \C m\sum_{y_s\in I_1}|h(x_i,y_s)|\,w(y_s)\varphi(y_s)+\C m^{-2\alpha-2}\sum_{y_s\in I_2}|h(x_i,y_s)|\\
        &\quad + \C m^{-2\beta-2}\sum_{y_s\in I_3}|h(x_i,y_s)|\\
        &\le \frac \C N\sum_{y_s\in I_1}|h(x_i,y_s)|\,w(y_s)\varphi(y_s)+\C \left(m^{-2\alpha-2}+m^{-2\beta-2}\right)\\
        &\le \C\frac \pi N\sum_{s=1}^N|h(x_i,y_s)|\,w(y_s)\varphi(y_s)+o(1),\qquad \quad i=1,\ldots,n+1,
    \end{align*}
    and \eqref{tesi-prop43} follows by applying \eqref{Gauss-conv} to the last Gauss--Chebyshev quadrature sum (
    $\zeta(y)=h(x,y)w(y)\varphi(y)$ is bounded and Riemann integrable in $L^1_v$) that yields
    \[
    \frac\pi N\sum_{s=1}^N |h(x_i,y_s)|\,w(y_s)\varphi(y_s)\le \C \int_{-1}^1 |h(x_i,y)|w(y)\,dy\le \C \max_x \int_{-1}^1 |h(x,y)|w(y)\,dy=\bigO(1),
    \]
    where the existence of the maximum follows from the Extreme Value Theorem, since \eqref{hp1-h}--\eqref{hp2-h} imply that $\Phi(x)=\int_{-1}^1 |h(x,y)|w(y)dy$ is a continuous function on $[-1,1]$.
\end{proof}

For the approximate equation \eqref{FIE-rdc} we state the analogous of Thm. \ref{th-collo}.
\begin{theorem}\label{th-rdc}
	Assume that $\operatorname{ker}(I-H)=\{0\}$ and let the hypotheses of Proposition \ref{prop-Htilde} be satisfied. Moreover, let $f^*\in C^0$ be the unique solution of \eqref{FIE} for a given $g\in C^0$. Then, for all sufficiently large $n\in\NN$ and $0<m<n$ such that $n\sim m\sim (n-m)$, the approximate equation \eqref{FIE-rdc} admits a unique stable solution $\tilde f_n^m$, which belongs to the space $\R_n^m$, and satisfies
    \begin{equation}\label{lim-rdc}
		\lim_{n\to \infty}
        \|f^*-\tilde f^m_n\|_\infty =0.
	\end{equation}
Furthermore, 
	\begin{equation}\label{lim-cond-rdc}
		\lim_{ n\to\infty}
        \kappa(I-\tilde H^m_n)=\kappa(I-H).
	\end{equation}
\end{theorem}
\begin{proof}
If we prove that
\begin{equation}\label{tesi-rdc}
  \lim_{n\to\infty}\|H-\tilde H_n^m\|=0,\qquad \forall m\sim n\sim (n-m), 
\end{equation}
then the proof follows analogously to that of Theorem   \ref{th-collo}.

To prove \eqref{tesi-rdc}, we note that
\[
    \|H-\tilde H_n^m\|\leq \|H-H_n^m\|+\|H_n^m-\tilde H_n^m\|,
\]
where, by Theorem \ref{th-collo}, the first addendum tends to zero as $n\to\infty$ and $m\sim n\sim (n-m)$. Hence. let us prove that the same holds for the second addendum.

Recalling the definitions (cf. \eqref{Hnm}, \eqref{tilde-Hnm}, \eqref{LC-Sum}) and using \eqref{bound-R}, $\forall \zeta\in C^0 $, we have
\begin{align*}
    \|H_n^m\zeta-\tilde H_n^m \zeta\|_\infty &= 
\max_{|x|\le 1}\left|\sum_{i=1}^{n+1}r_{n,i}^m(x)\int_{-1}^1\left[h(x_i,y)-V_N^MF_i(y)k_m(y,y)\right] \zeta(y)w(y)dy\right|\\   
    &\leq    \|R_n^m\| \|\zeta\|_\infty \max_{1\leq i\leq n+1}
    \int_{-1}^1 \left|h(x_i,y)-V_N^MF_i(y)k_m(y,y)\right| w(y)dy\\
    &\leq    \C \|\zeta\|_\infty \max_{1\leq i\leq n+1}
    \int_{-1}^1 \left|h(x_i,y)-V_N^MF_i(y)k_m(y,y)\right| w(y)dy,
\end{align*}
Consequently, due to the assumptions on $h$, the statement follows by applying \eqref{lim-VP1} with  $f(y)=h(x_i,y)$.
\end{proof}

\subsection{The associated  linear system}
In the following, we suppose that \eqref{FIE-rdc} has a unique solution $\tilde f_n^m$ and see how to compute it. To this aim, since $\tilde f_n^m\in \R_n^m$, we  write 
\begin{equation}\label{ftilde}
    \tilde f_n^m(x):=\sum_{j=1}^{n+1}\tilde f_j r_{n,j}(x), \qquad |x|\leq 1,
\end{equation}
and compute the unknown coefficients $\tilde f_j=\tilde f_n^m(x_j)$, by solving the linear system resulting from the collocation
\begin{equation}\label{collo}
    \tilde f_n^m(x_j) -\tilde H_n^m \tilde f_n^m(x_j)=g_n^m(x_j), \qquad j=1,\ldots, n+1. 
\end{equation}
Let us derive a more  explicit form of this linear system. Firstly, note that, by \eqref{interp},
\begin{equation}\label{gnm-j}
    g_n^m(x_j)=R_n^mg(x_j)=g(x_j), \qquad j=1,\ldots, n+1.
\end{equation} 
On the other hand, by definition (cf. \eqref{tilde-Hnm},  \eqref{ftilde} and \eqref{rnm}),  we have
\begin{align*}
    \tilde H_n^m\tilde f_n^m(x_j)&=\int_{-1}^1 V_N^M F_j(y)k_m(y,y)\tilde f_n^m(y)w(y)dy\\
    &= \int_{-1}^1 V_N^M F_j(y) k_m(y,y)\left[\sum_{i=1}^{n+1} \tilde f_i \lambda_i \frac{k_n(x_i,y)k_m(x_i,y)}{k_m(y,y)}\right] w(y)dy
    \\
    &= \sum_{i=1}^{n+1} \tilde f_i \lambda_i \int_{-1}^1  \left[ V_N^M F_j(y) k_m(x_i,y)\right] k_n(x_i,y) w(y)dy,
\end{align*}
where $V_N^M F_j(y) k_m(x_i,y)$ is a polynomial of degree at most $n$ for all indices $i$ and $j$ (due to \eqref{cond-1}). Hence, by the reproducing property \eqref{inva}, we obtain
\begin{equation}\label{Hnmj}
    \tilde H_n^m\tilde f(x_j)= \sum_{i=1}^{n+1} \tilde f_i  \lambda_i  V_N^M F_j(x_i)k_m(x_i,x_i), \qquad j=1,\ldots, n+1,
\end{equation}
where we point out that the polynomials $V_N^M F_j(x_i)$, $i,j=1,\ldots, n+1$ can be efficiently computed as discussed in Subsection \ref{subsec:choice_Polynomials} (item \ref{P5}). \\
It follows that the linear system \eqref{collo} can be expressed explicitly as
\begin{equation}\label{eq:linear_system_components}
	\tilde{f}_j - \sum_{i=1}^{n+1}\tilde{f}_i \lambda_i  V_N^M F_j(x_i) k_m(x_i,x_i)=g(x_j), \qquad j=1,\dots,n+1,
\end{equation}
i.e., setting $\bm{\tilde f} =[
			\tilde f_1, \ldots, \tilde f_{n+1}]^T$, $\bm{g} = [
			g(x_1), \ldots, g(x_{n+1})]^T$, and 
\begin{align*}
    \bm{D}&=\diag[\lambda_i k_m(x_i,x_i)]_{i=1,...,n+1}, \\
    \mathcal{V}&=[V_{j,i}]_{j,i=1,...,n+1}, \qquad
    V_{j,i}=V_N^MF_j(x_i), \qquad F_j(x)=\lambda_m(x)h(x_j,x),
    \qquad j,i=1,...,n+1,
\end{align*}
 system \eqref{eq:linear_system_components} can be compactly written as
\begin{equation}\label{eq:Metod_System}
	(\mathcal{I}-\mathcal{\tilde A})\bm{\tilde{f}}=\bm{g}, \qquad \text{where} \qquad
    \qquad \mathcal{\tilde A}=\mathcal{VD}. 
\end{equation}

\section{Numerical Tests}\label{sec:Numerical_Tests}
In this section, we present some test cases to illustrate the properties of the MN and RDC methods introduced in Sections \ref{sec:RC_Method} and \ref{sec:RDC_Method}, respectively. Let $\{\zeta_l\}_{1\leq l \leq N_\zeta}$ denote a uniform discretization of the domain $[-1,1]$ In what follows, we take $N_\zeta=10^3$ and adopt the notation RDC$_\theta$ to specify the value of $m=\left\lfloor{\theta n}\right\rfloor.$ We compare the proposed RDC and MN methods with the results obtained via the variants of the Nystr{\"o}m methods (see, e.g., \cite[Chapter 4]{Atkinson_1997} and \cite{Atkinson_Matlab}) obtained by using Gauss \cite{Gautschi}, anti Gauss \cite{Patricia,Laurie_1996}, averaged Gauss and weighted averaged Gauss \cite{Luisa_Avg,SpalevicI,SpalevicII} quadrature rules.
To assess the convergence behavior of the RDC method, we define the uniform and mean errors as follows
\begin{equation}\label{eq:Errors}
    \varepsilon^n(\zeta_l)=\left|f^{ref}(\zeta_l)-\tilde{f}^n(\zeta_l)\right|, \qquad\qquad 
    E^n_{\infty}=\max_l \varepsilon^n(\zeta_l), \qquad\qquad 
    \mathbb{E}^n=\dfrac{1}{N_\zeta}\sum_{l=1}^{N_\zeta} \varepsilon^n(\zeta_l),
\end{equation}
where $\tilde{f}^n$ denotes the approximate solution achieved by using $n+1$ nodes, and $f^{ref}(\zeta_l)$ denotes the reference solution, which is either known analytically or computed symbolically using Wolfram Mathematica with high precision. 
The comparison is conducted primarily in terms of the errors defined in \eqref{eq:Errors}, while also taking into account the condition number $\kappa_\infty$ of the linear system arising from each method. All the simulations have been performed using MATLAB R2025a on an Intel(R) Core(TM) i9-14900KF processor operating at 3200 MHz.

\begin{test}\label{test:Ex_0_12} 
With the aim of investigating the convergence properties of the MN and RDC methods, we consider the FIE~\eqref{FIE} with
\begin{equation}\label{eq:Ex_0_Details}
   w(x) = \dfrac{1}{\sqrt[4]{(1 - x^2)}}, \qquad\qquad h(x,y) = \left(\dfrac{1 - x^2}{\sin(4\pi y) + 2} \right) \cdot 
    \begin{cases}
       0, & \text{for }  y = 1,  \\
         y, & \text{for } y \in  [0,1),  \\
        1  + y, & \text{for } y \in  [-1,0).
    \end{cases}
\end{equation}
The function $g(x)=-1 + x^2 + (1 - x^2)^{1/4} \, \bigl( 2 + \sin(4\pi x) \bigr)$ is chosen so that the exact solution is given by $f^*(x) = ( 2 + \sin(4\pi x) ) \sqrt[4]{(1 - x^2)^3}$. \newline Figures~\ref{fig:Ar_Ex_12_Bis_Results_MN} and~\ref{fig:Ar_Ex_12_Bis_Results_RDC} report the errors obtained with the MN scheme~\eqref{Nystrom-sist}--\eqref{fnm-MN} and the RDC method~\eqref{ftilde}--\eqref{eq:Metod_System}, respectively, for several values of $\theta$ and increasing $n$. 
The MN scheme exhibits a linear convergence, with the error decaying as $n^{-1}$, and yields results comparable to those of the classical Nystr\"om method (based on the same Gaussian rule), especially for large values of~$n$.
\begin{figure}[tbp]
    \centering
    \includegraphics[width=0.8\linewidth]{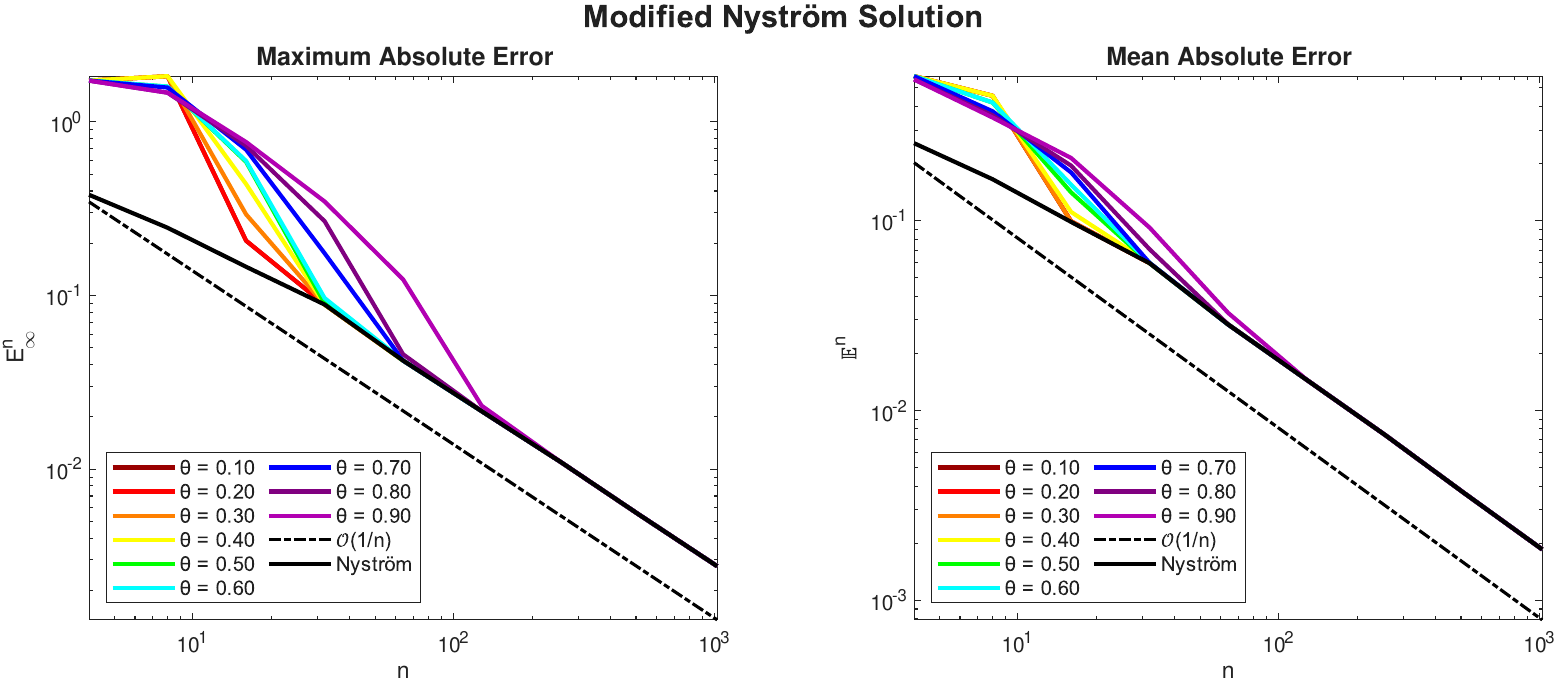}
    \caption{Experimental convergence of the MN scheme for Test~\ref{test:Ex_0_12}.}
    \label{fig:Ar_Ex_12_Bis_Results_MN}
\end{figure}
On the contrary, the RDC method achieves an error decay of approximately $n^{-2}$ for almost all values of~$\theta$, significantly outperforming both the classical and the Modified Nystr\"om schemes. Table~\ref{tab:Ar_Ex_0_Ord_Convergence} reports the experimental orders of convergence in the uniform norm, computed as
\begin{equation*}
    p_\theta=-\log_2\!\left(\frac{E_\infty^n}{E_\infty^{2n}}\right),
\end{equation*}
together with the mean convergence rate $\mathbb{E}(p_\theta)$.
\begin{figure}[tbp]
    \centering
    \includegraphics[width=0.8\linewidth]{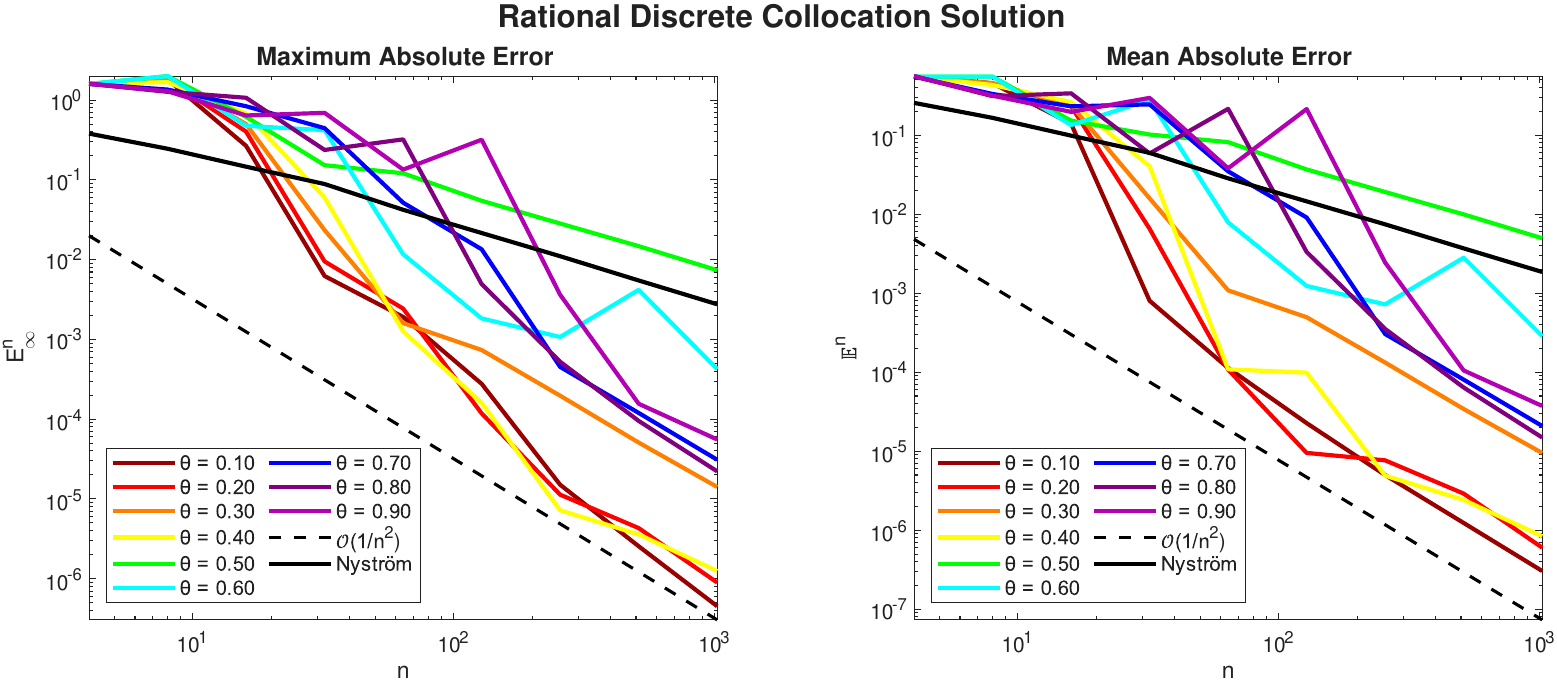}
    \caption{Experimental convergence of the RDC method for Test~\ref{test:Ex_0_12}.}
    \label{fig:Ar_Ex_12_Bis_Results_RDC}
\end{figure}
\begin{table}[htbp]
    \centering
    \caption{Test~\ref{test:Ex_0_12}: Experimental orders of convergence $p_\theta$  of RDC method 
    computed for different values of~$\theta$.}
    \label{tab:Ar_Ex_0_Ord_Convergence}
    \footnotesize
    \setlength{\tabcolsep}{3.4pt}
    \begin{tabular}{
    c
    *{9}{S[table-format=-1.2]}
    }
    \toprule
    {$n$}
    & {$\mathrm{RDC}_{0.10}$}
    & {$\mathrm{RDC}_{0.20}$}
    & {$\mathrm{RDC}_{0.30}$}
    & {$\mathrm{RDC}_{0.40}$}
    & {$\mathrm{RDC}_{0.50}$}
    & {$\mathrm{RDC}_{0.60}$}
    & {$\mathrm{RDC}_{0.70}$}
    & {$\mathrm{RDC}_{0.80}$}
    & {$\mathrm{RDC}_{0.90}$} \\
    \midrule
    4    & -0.19 & -0.19 & -0.19 & -0.08 & -0.32 & -0.32 &  0.23 & 0.33 & 0.31 \\
    8    &  2.77 &  2.18 &  1.86 &  1.31 &  1.70 &  2.07 &  0.70 & 0.25 & 1.01 \\
    16   &  5.42 &  5.41 &  4.42 &  3.51 &  2.01 &  0.18 &  0.91 & 2.17 &-0.11 \\
    32   &  1.70 &  1.96 &  3.87 &  5.57 &  0.34 &  5.17 &  3.10 &-0.44 & 2.36 \\
    64   &  2.79 &  4.37 &  1.12 &  2.98 &  1.15 &  2.67 &  1.94 & 6.03 &-1.25 \\
    128  &  4.20 &  3.38 &  1.90 &  4.47 &  0.95 &  0.77 &  4.90 & 3.24 & 6.46 \\
    256  &  2.56 &  1.39 &  1.95 &  1.00 &  0.94 & -1.96 &  1.92 & 2.46 & 4.54 \\
    512  &  2.50 &  2.27 &  1.85 &  1.53 &  1.00 &  3.30 &  1.95 & 2.10 & 1.49 \\
    1024 &  2.00 &  2.12 &  1.86 &  1.42 &  1.01 &  2.36 &  2.12 & 1.98 & 0.19 \\
    2048 &  2.00 &  1.93 &  1.90 &  1.39 &  1.00 &  0.50 &  2.04 & 1.98 & 2.10 \\
    4096 &  2.00 &  1.98 &  1.91 &  1.58 &  1.00 & -2.15 &  1.98 & 2.00 & 2.05 \\
    \midrule
    {$\mathbb{E}(\bar{\rho}_{\theta})$}
    & 2.52 & 2.44 & 2.04 & 2.24 & 0.98 & 1.14 & 1.98 & 2.01 & 1.74 \\
    \bottomrule
    \end{tabular}
\end{table}

Finally, to provide a quantitative comparison between the MN and RDC schemes, we consider the error ratio
\begin{equation*}
    \mathcal{R}_{n,\theta}^{\mathrm{mean}}
    =
    \frac{\mathbb{E}^n_{\mathrm{RDC}}}
         {\mathbb{E}^n_{\mathrm{MN}}},
\end{equation*}
where values smaller than one indicate a lower mean error for the RDC scheme. As reported in Table~\ref{tab:Ar_Ex_0_Log_Ratio_Mean}, the RDC method consistently outperforms the MN scheme, yielding error reductions of up to two orders of magnitude, particularly for larger values of~$\theta$. Since both methods rely on the same reconstruction operator (see~\eqref{fnm-MN} and~\eqref{ftilde}), these improvements can be attributed to the integral discretization strategy rather than to the reconstruction procedure.
\begin{table}[htbp]
    \centering
    \caption{Error ratio between the RDC and MN methods for Test~\ref{test:Ex_0_12}. Smaller values indicate a reduced mean error for the RDC method.}
    \label{tab:Ar_Ex_0_Log_Ratio_Mean}
    \footnotesize
    \setlength{\tabcolsep}{4pt}
    \begin{tabular}{c*{9}{l}}
    \toprule
    {$n$}
    & {$\theta=0.10$}
    & {$\theta=0.20$}
    & {$\theta=0.30$}
    & {$\theta=0.40$}
    & {$\theta=0.50$}
    & {$\theta=0.60$}
    & {$\theta=0.70$}
    & {$\theta=0.80$}
    & {$\theta=0.90$} \\
    \midrule
    4
    & $9.5\cdot10^{-1}$ & $1.0\cdot10^{0}$ & $1.4\cdot10^{0 \phantom{-}}$ & $1.3\cdot10^{-2}$ & $4.0\cdot10^{-3}$ & $1.6\cdot10^{-3}$ & $7.0\cdot10^{-4}$ & $3.0\cdot10^{-4}$ & $2.0\cdot10^{-4}$ \\
    8
    & $9.5\cdot10^{-1}$ & $1.0\cdot10^{0 \phantom{-}}$ & $2.3\cdot10^{0 \phantom{-}}$ & $1.1\cdot10^{-1}$ & $3.8\cdot10^{-3}$ & $7.0\cdot10^{-4}$ & $1.0\cdot10^{-3}$ & $8.0\cdot10^{-4}$ & $3.0\cdot10^{-4}$ \\
    16
    & $9.5\cdot10^{-1}$ & $1.0\cdot10^{0 \phantom{-}}$ & $2.4\cdot10^{0 \phantom{-}}$ & $2.7\cdot10^{-1}$ & $3.8\cdot10^{-2}$ & $3.4\cdot10^{-2}$ & $1.8\cdot10^{-2}$ & $9.3\cdot10^{-3}$ & $5.1\cdot10^{-3}$ \\
    32
    & $9.5\cdot10^{-1}$ & $9.5\cdot10^{-1}$ & $2.4\cdot10^{0 \phantom{-}}$ & $6.8\cdot10^{-1}$ & $3.8\cdot10^{-3}$ & $6.8\cdot10^{-3}$ & $7.0\cdot10^{-4}$ & $7.0\cdot10^{-4}$ & $5.0\cdot10^{-4}$ \\
    64
    & $9.5\cdot10^{-1}$ & $1.3\cdot10^{0 \phantom{-}}$ & $1.1\cdot10^{0 \phantom{-}}$ & $1.7\cdot10^{0 \phantom{-}}$ & $2.9\cdot10^{0 \phantom{-}}$ & $2.5\cdot10^{0 \phantom{-}}$ & $2.6\cdot10^{0 \phantom{-}}$ & $2.7\cdot10^{0 \phantom{-}}$ & $2.7\cdot10^{0 \phantom{-}}$ \\
    128
    & $9.5\cdot10^{-1}$ & $1.3\cdot10^{0 \phantom{-}}$ & $8.7\cdot10^{-1}$ & $4.6\cdot10^{0 \phantom{-}}$ & $2.8\cdot10^{-1}$ & $8.5\cdot10^{-2}$ & $9.7\cdot10^{-2}$ & $7.6\cdot10^{-1}$ & $1.5\cdot10^{-1}$ \\
    256
    & $9.5\cdot10^{-1}$ & $8.8\cdot10^{-1}$ & $1.3\cdot10^{0 \phantom{-}}$ & $4.1\cdot10^{0 \phantom{-}}$ & $1.2\cdot10^{0 \phantom{-}}$ & $6.2\cdot10^{-1}$ & $4.1\cdot10^{-2}$ & $2.2\cdot10^{-2}$ & $1.1\cdot10^{-2}$ \\
    512
    & $1.0\cdot10^{0 \phantom{-}}$ & $8.7\cdot10^{-1}$ & $1.7\cdot10^{0 \phantom{-}}$ & $8.3\cdot10^{-1}$ & $7.6\cdot10^{0 \phantom{-}}$ & $2.3\cdot10^{-1}$ & $4.8\cdot10^{-2}$ & $1.7\cdot10^{-2}$ & $8.0\cdot10^{-3}$ \\
    1024
    & $1.0\cdot10^{0 \phantom{-}}$ & $9.0\cdot10^{-1}$ & $9.2\cdot10^{-1}$ & $3.2\cdot10^{0 \phantom{-}}$ & $1.2\cdot10^{0 \phantom{-}}$ & $1.5\cdot10^{1}$ & $3.3\cdot10^{-1}$ & $2.9\cdot10^{-2}$ & $2.0\cdot10^{-2}$ \\
    \bottomrule
    \end{tabular}
\end{table}

\end{test}

\begin{test}\label{test:Ex_1} 
Our second example concerns the FIE 
of the form \eqref{FIE} with
\begin{equation}\label{eq:Ex_1_Details}
    \begin{split}
        &w(x)=1, \qquad \ \ \ g(x) = \left( \dfrac{13442827}{6750} + \dfrac{167827}{225}x^2 \right) \cdot 10^{-3},  
        \\
        & h(x,y)=\dfrac{(30x^2 + 1)|4y^2 - 1|}{360}  \cdot \left( \chi_{1}(y)-\chi_{2}(y)+2\chi_{3}(y)-\dfrac{1}{2}\chi_{4}(y)+\dfrac{1}{2}\chi_{5}(y) \right),
    \end{split}
\end{equation}
where $\chi_{i}(\cdot)$ denotes the characteristic function of the interval $I_i,$ $i=1,\dots,5,$ defined as follows
\begin{equation*}
    I_1=[-1,-3/5], \quad I_2=[-2/5,-1/5], \quad I_3=[-1/10,1/10], \quad I_4=[1/5,2/5], \quad I_5=[3/5,1]. 
\end{equation*}
In this setting, the exact solution 
is $f^*(x) = 2 + x^2.$ The uniform and mean errors associated with the different methods, computed according to \eqref{eq:Errors}, are summarized in Table~\ref{tab:Ar_Ex_1}. Further insight into the behavior of the approximation errors $\varepsilon^n(\zeta_l)$ is provided in Figure~\ref{fig:Ar_Ex_1_Results} for the representative cases $n=4$ and $n=1024$ .

\begin{figure}[htbp]
        \centering
        \includegraphics[width=0.8\linewidth]{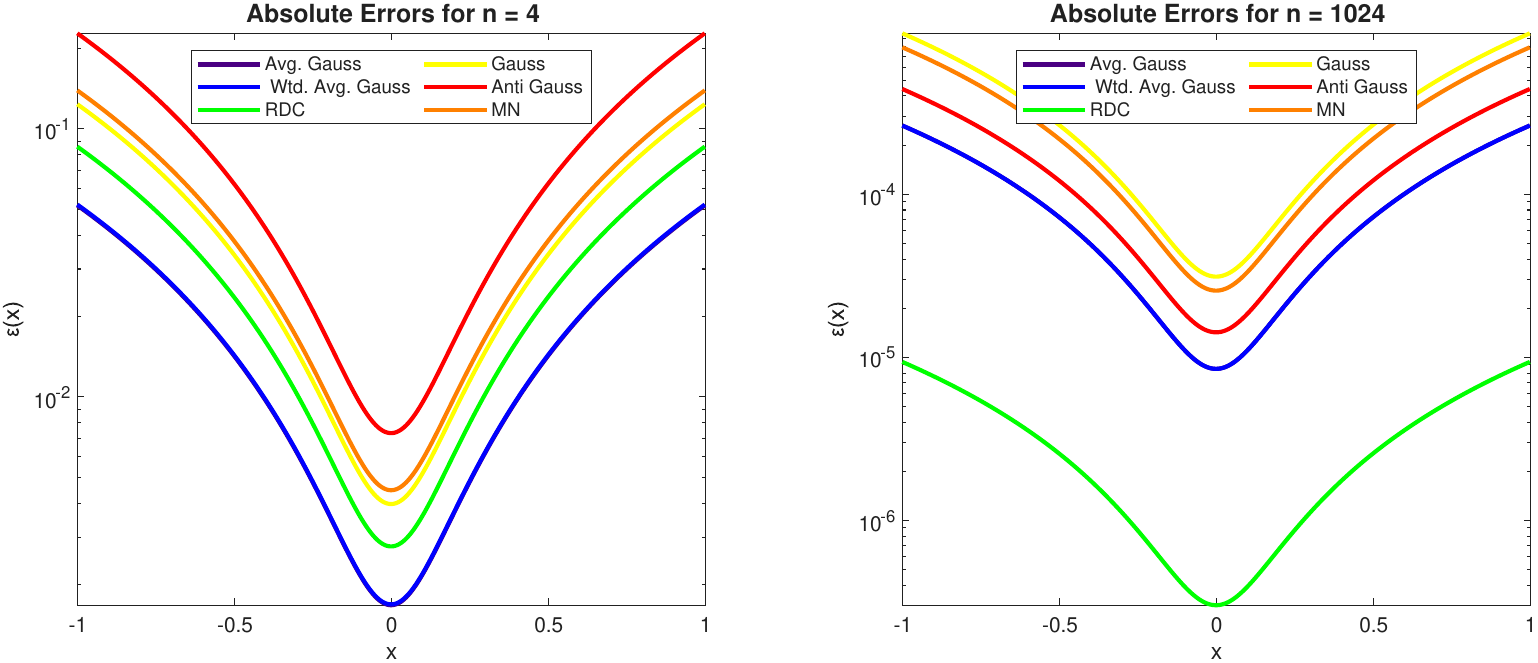}
        \caption{Test~\ref{test:Ex_1} : Pointwise absolute errors achieved for $n=4$ (left) and $n=1024$ (right) by the RDC and MN schemes, compared with Nystr\"om variants corresponding to different quadrature rules on the same number of nodes.}
      \label{fig:Ar_Ex_1_Results}
\end{figure}
    
\begin{table}[htbp]
    \centering
    \caption{Comparison of different methods applied to Test~\ref{test:Ex_1}.}
    \label{tab:Ar_Ex_1}
    \small
    \begin{tabular}{cc
    S[table-format=1.2e-1]
    S[table-format=1.2e-1]
    S[table-format=1.2e-1]
    S[table-format=1.2e-1]
    S[table-format=1.2e-1]
    S[table-format=1.2e-1]}
    \toprule
    {$n$} & {Quantity}
    & {RDC}
    & {MN}
    & {Gauss}
    & {Anti-Gauss}
    & {Avg. Gauss}
    & {Weighted Avg.}\\
    \midrule
    4
    & $E_\infty^n$
    & 8.58e-2 & 1.39e-01 & 1.24e-1 & 2.27e-1 & 5.18e-2 & 5.22e-2 \\
    & $\mathbb{E}^n$
    & 3.05e-2 & 4.95e-02 & 4.39e-2 & 8.08e-2 & 1.84e-2 & 1.86e-2  \\
    & $\kappa_\infty$
    & 1.21 & 1.26 & 1.17 & 1.42 & 1.35 & 1.20 \\
    \midrule
    16
    & $E_\infty^n$
    & 8.83e-3 & 7.38e-03 & 6.98e-2 & 8.62e-3 & 3.92e-2 & 3.92e-2 \\
    & $\mathbb{E}^n$
    & 3.14e-3 & 2.62e-03 & 2.48e-2 & 3.06e-3 & 1.39e-2 & 1.39e-2 \\
    & $\kappa_\infty$
    & 1.32 & 1.27 & 1.36 & 1.27 & 1.32 & 1.13\\
    \midrule
    64
    & $E_\infty^n$
    & 9.18e-4 & 1.71e-02 & 1.19e-4 & 1.73e-2 & 8.72e-3 & 8.72e-3 \\
    & $\mathbb{E}^n$
    & 3.26e-4 & 6.09e-03 & 4.22e-5 & 6.16e-3 & 3.10e-3 & 3.10e-3 \\
    & $\kappa_\infty$
    & 1.31 & 1.30 & 1.29 & 1.31 & 1.30 & 1.15 \\
    \midrule
    256
    & $E_\infty^n$
    & 2.04e-4 & 3.10e-03 & 3.92e-3 & 3.12e-3 & 3.97e-4 & 3.97e-4 \\
    & $\mathbb{E}^n$
    & 7.24e-5 & 1.10e-03 & 1.39e-3 & 1.11e-3 & 1.41e-4 & 1.41e-4 \\
    & $\kappa_\infty$
    & 1.30 & 1.30 & 1.29 & 1.30 & 1.30 & 1.14 \\
    \midrule
    1024
    & $E_\infty^n$
    & 9.37e-6 & 7.95e-04 & 9.68e-4 & 4.41e-4 & 2.64e-4 & 2.64e-4 \\
    & $\mathbb{E}^n$
    & 3.33e-6 & 2.83e-04 & 3.44e-4 & 1.57e-4 & 9.37e-5 & 9.37e-5 \\
    & $\kappa_\infty$
    & 1.30 & 1.30 & 1.30 & 1.29 & 1.30 & 1.14 \\
    \bottomrule
    \end{tabular}
\end{table}

\end{test}

\begin{test}\label{test:Ex_2} 
Here we address
the FIE \eqref{FIE} with 
\begin{equation}\label{eq:Ex_2_Details}
    \begin{split}
            & w(x)=1-x^2, \quad h(x,y)=\dfrac{W(x)W(y)}{4}, \quad \text{where} \quad  W(x) := \sum_{n=0}^{10} \dfrac{1}{2^n}\cos(5^n \pi x), \\
            & g(x)=(x+3)(x+2)^{-1} - \gamma W(x), \quad \text{where} \quad \gamma:=\int_{-1}^{1}\dfrac{(y^3+3y^2-y-3)W(y)}{4(y+2)} \ dy. 
    \end{split}
\end{equation}
The exact solution is $f^*(x)=(x+3)(x+2)^{-1}.$ In this case we remark that, although $h(x,y) \in C^\infty([-1,1]^2)$, the approximation of the solution 
is demanding due to the oscillatory behavior of the kernel. Nonetheless, the RDC method proves effective in this setting, leading to smoother solutions for small values of $n$ (see Figure~\ref{fig:Ar_Ex_2_Results}) and, more generally, to higher accuracy than that achieved by the other methods under consideration, as shown in Table~\ref{tab:Ar_Ex_2}.

\begin{figure}[htbp]
        \centering
        \includegraphics[width=0.8\linewidth]{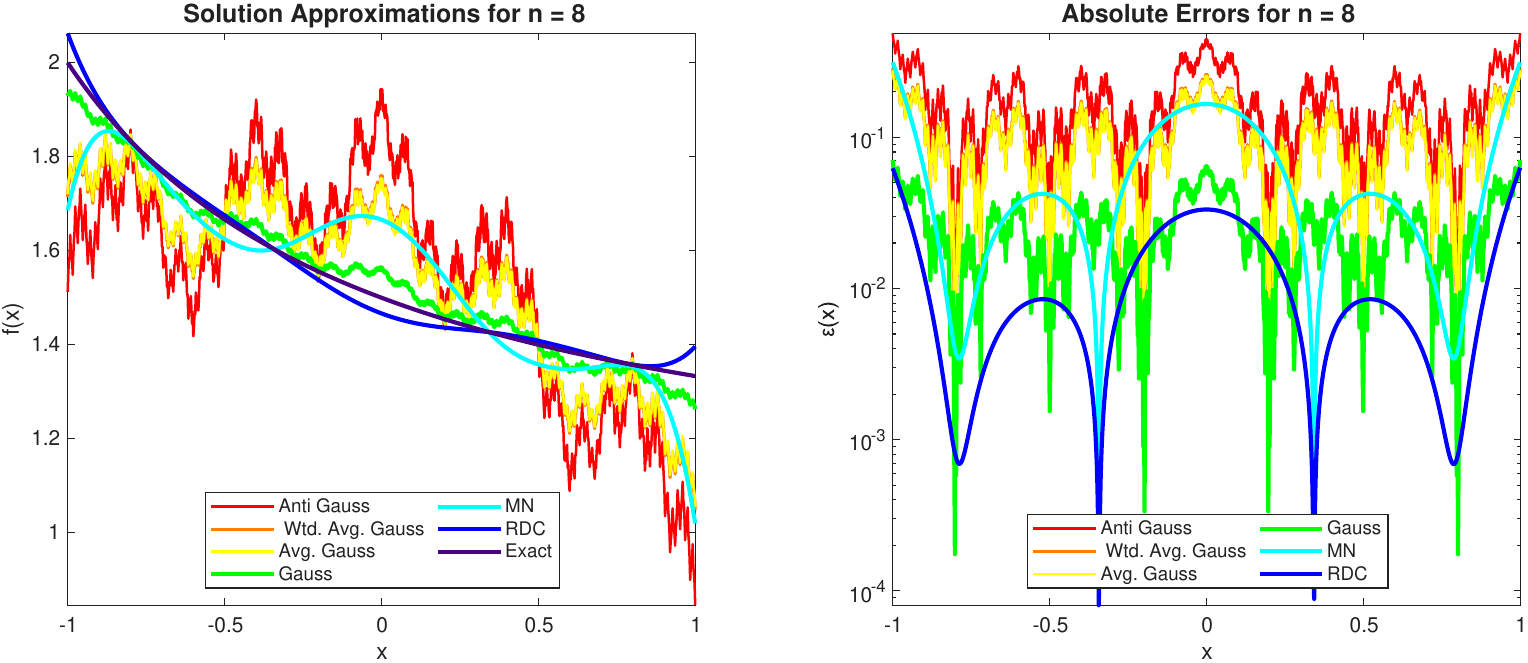}
        \caption{Approximation outcomes and errors for Test \ref{test:Ex_2} with $n=8.$ The RDC and MN schemes are compared with Nystr\"om variants based on different quadrature rules.}
      \label{fig:Ar_Ex_2_Results}
\end{figure}

\begin{table}[htbp]
    \centering
    \caption{Comparison of different methods applied to Test~\ref{test:Ex_2}.}
    \label{tab:Ar_Ex_2}
    \small
    \begin{tabular}{cc
    S[table-format=1.2e-1]
    S[table-format=1.2e-1]
    S[table-format=1.2e-1]
    S[table-format=1.2e-1]
    S[table-format=1.2e-1]
    S[table-format=1.2e-1]}
    \toprule
    {$n$} & {Quantity}
    & {RDC}
    & {MN}
    & {Gauss}
    & {Anti-Gauss}
    & {Avg. Gauss}
    & {Weighted Avg.} \\
    \midrule
    
    4 & $E_\infty^n$
    & 6.30e-2 & 1.14e+00 & 7.07e-2 & 4.88e-1 & 2.79e-1 & 2.86e-1 \\
    & $\mathbb{E}^n$
    & 1.30e-2 & 4.98e-01 & 2.51e-2 & 1.73e-1 & 9.91e-2 & 1.01e-1 \\
    & $\kappa_\infty$
    & 1.17 & 2.78 & 1.66 & 4.08 & 2.65 & 1.92 \\
    
    \midrule
    
    32 & $E_\infty^n$
    & 2.46e-3 & 1.64e-01 & 1.31e-1 & 1.19e-1 & 5.87e-3 & 6.28e-3 \\
    & $\mathbb{E}^n$
    & 8.73e-4 & 5.82e-02 & 4.64e-2 & 4.22e-2 & 2.09e-3 & 2.23e-3 \\
    & $\kappa_\infty$
    & 2.17 & 2.16 & 2.09 & 2.23 & 2.24 & 1.51 \\
    
    \midrule
    
    128 & $E_\infty^n$
    & 7.49e-4 & 1.43e-03 & 4.18e-3 & 2.94e-3 & 6.17e-4 & 6.17e-4 \\
    & $\mathbb{E}^n$
    & 2.54e-4 & 4.85e-04 & 1.48e-3 & 1.05e-3 & 2.19e-4 & 2.19e-4 \\
    & $\kappa_\infty$
    & 2.28 & 2.26 & 2.19 & 2.27 & 2.26 & 1.53 \\
    
    \midrule
    
    512 & $E_\infty^n$
    & 2.86e-4 & 2.79e-03 & 6.29e-4 & 9.34e-4 & 7.82e-4 & 7.82e-4 \\
    & $\mathbb{E}^n$
    & 1.02e-4 & 9.93e-04 & 2.23e-4 & 3.32e-4 & 2.77e-4 & 2.77e-4 \\
    & $\kappa_\infty$
    & 2.31 & 2.28 & 2.26 & 2.28 & 2.27 & 1.54 \\
    
    \midrule
    
    2048 & $E_\infty^n$
    & 3.50e-5 & 1.13e-03 & 1.65e-3 & 1.32e-3 & 1.48e-3 & 1.48e-3 \\
    & $\mathbb{E}^n$
    & 1.24e-5 & 4.01e-04 & 5.85e-4 & 4.69e-4 & 5.27e-4 & 5.27e-4 \\
    & $\kappa_\infty$
    & 2.29 & 2.28 & 2.28 & 2.28 & 2.28 & 1.54 \\
    
    \bottomrule
    \end{tabular}
\end{table}

\end{test}

\begin{test}\label{test:Ex_3} 
    Let us consider the following two-points boundary value problem \cite{Allouch_2024}
    \begin{equation}\label{eq:BVP}
       \begin{cases}
           f^{\prime \prime}(x)+ f(x)\left(1+\dfrac{\pi^2}{4}\right)-\cos\!\left(\dfrac{\pi x}{2}\right)=0, \qquad \qquad x\in (-1,1),\\
           f(-1)=f(1), \qquad \quad \qquad \! f^{\prime}(-1)=f^{\prime}(1),
       \end{cases}  
    \end{equation}
    which is equivalent to a FIE of the form \eqref{FIE} with Legendre weight ($\alpha=\beta=0$) and
    \begin{equation}\label{eq:Ex_3_Details}
        g(x)=\dfrac{1}{\pi^2} \left(2\cos\!\left(\frac{\pi x}{2}\right) + \pi x \sin\!\left(\frac{\pi x}{2}\right)\right), \qquad \qquad h(x,y)=-\dfrac{1}{\pi}\sin\!\left(\dfrac{\pi |x-y|}{2}\right).
    \end{equation}
   The exact solution at the boundary points takes the following values
    \begin{equation*}
        f^*(\pm 1) = -\dfrac{\pi}{\sqrt{4+\pi^2}} \, \cot\!\left(\dfrac{\sqrt{4+\pi^2}}{2}\right).
    \end{equation*}
    The absolute approximation errors $\varepsilon^n(\pm 1)$  obtained by the different methods  are plotted in Figure~\ref{fig:Ar_Ex_3_Results}. Additionally, the experimental convergence rates $\rho^n = \log_2\left(\varepsilon^{n/2}(\pm 1)/\varepsilon^{n}(\pm 1)\right)$ are reported in Table~\ref{tab:Ar_Ex_3}. The results confirm that all methods exhibit at least second-order convergence, with the RDC scheme yielding more accurate results.

    \begin{figure}[htbp]
        \centering
        \includegraphics[width=0.8\linewidth]{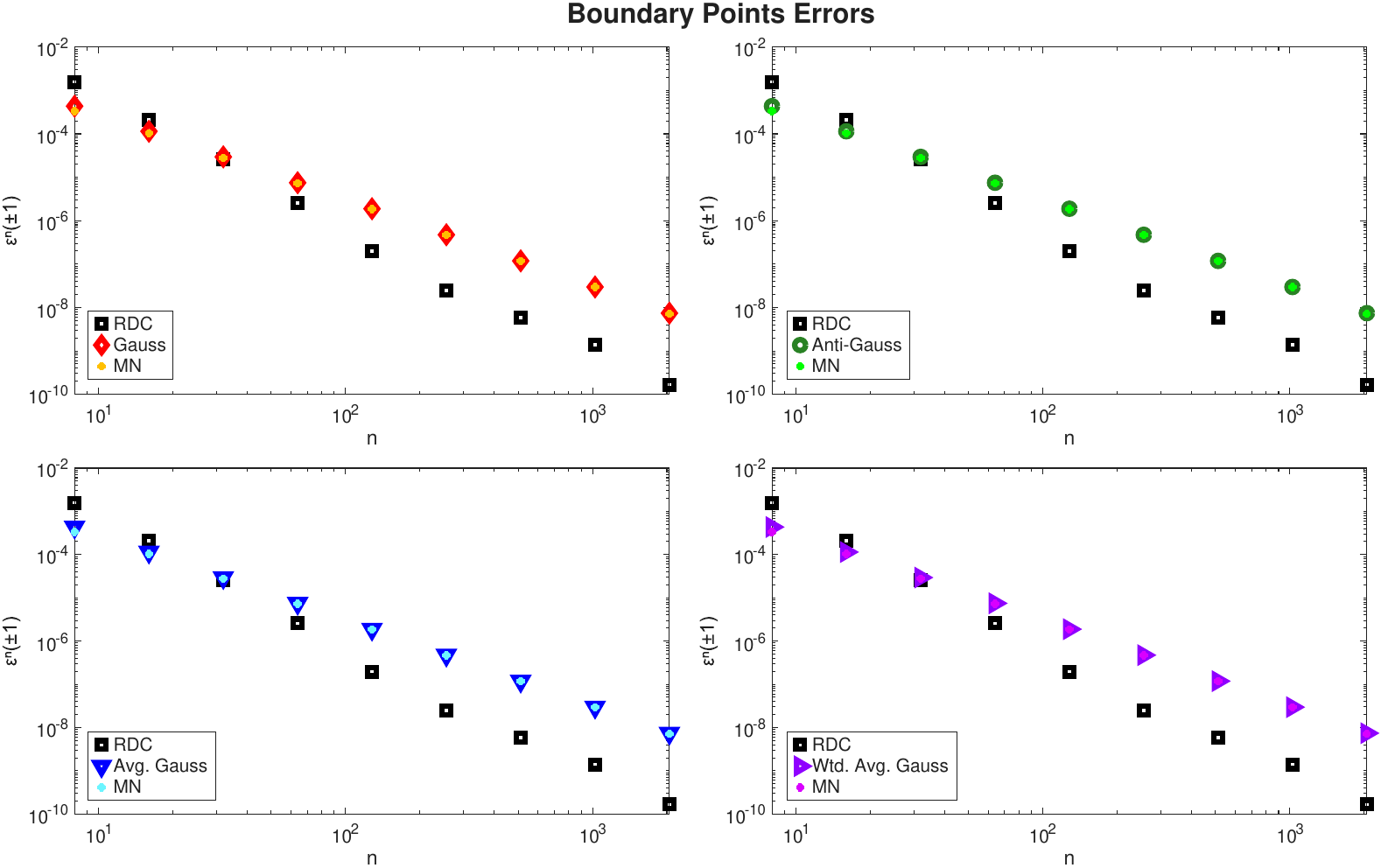}
        \caption{Absolute errors at the boundary nodes for Test~\ref{test:Ex_3}. The RDC and MN schemes are compared with Nystr\"om variants based on different quadrature rules.}
        \label{fig:Ar_Ex_3_Results}
    \end{figure}

    \begin{table}[htbp]
        \centering
        \caption{Comparison of different methods applied to Test~\ref{test:Ex_3}. Since the Nystr\"om methods based on the Gauss, Anti-Gauss, Avg. Gauss, and Weighted Avg. Gauss quadrature rules yield identical results, they are reported in a single column.}
        \label{tab:Ar_Ex_3}
        \small
        \begin{tabular}{cc
        S[table-format=1.2e-1]
        S[table-format=1.2e-1]
        S[table-format=1.2e-1]}
        \toprule
        {$n$} & {Quantity}
        & {RDC}
        & {MN}
        & {Nystr\"om variants} \\
        \midrule
        8
        & $\varepsilon^n(\pm1)$
        & 1.57e-3
        & 3.36e-4
        & 4.34e-4 \\
        & $\rho^n$
        & \multicolumn{1}{c}{---}
        & \multicolumn{1}{c}{---}
        & \multicolumn{1}{c}{---} \\
        \midrule
        16
        & $\varepsilon^n(\pm1)$
        & 2.09e-4
        & 1.03e-4
        & 1.15e-4 \\
        & $\rho^n$
        & 2.91
        & 1.70
        & 1.92 \\
        \midrule
        32
        & $\varepsilon^n(\pm1)$
        & 2.59e-5
        & 2.79e-5
        & 2.95e-5 \\
        & $\rho^n$
        & 3.01
        & 1.89
        & 1.96 \\
        \midrule
        64
        & $\varepsilon^n(\pm1)$
        & 2.57e-6
        & 7.26e-6
        & 7.48e-6 \\
        & $\rho^n$
        & 3.34
        & 1.94
        & 1.98 \\
        \midrule
        128
        & $\varepsilon^n(\pm1)$
        & 1.97e-7
        & 1.86e-6
        & 1.89e-6 \\
        & $\rho^n$
        & 3.70
        & 1.97
        & 1.99 \\
        \midrule
        256
        & $\varepsilon^n(\pm1)$
        & 2.46e-8
        & 4.69e-7
        & 4.73e-7 \\
        & $\rho^n$
        & 3.00
        & 1.98
        & 2.00 \\
        \midrule
        512
        & $\varepsilon^n(\pm1)$
        & 5.90e-9
        & 1.18e-7
        & 1.19e-7 \\
        & $\rho^n$
        & 2.06
        & 1.99
        & 2.00 \\
        \midrule
        1024
        & $\varepsilon^n(\pm1)$
        & 1.40e-9
        & 2.96e-8
        & 2.97e-8 \\
        & $\rho^n$
        & 2.07
        & 2.00
        & 2.00 \\
        \midrule
        2048
        & $\varepsilon^n(\pm1)$
        & 1.66e-10
        & 7.12e-9
        & 7.42e-9 \\
        & $\rho^n$
        & 3.07
        & 2.05
        & 2.00 \\
        \bottomrule
        \end{tabular}
    \end{table}

\end{test}

\section{Conclusions}\label{sec:conclusions}
For the numerical solution of Fredholm integral equations of the second kind, we developed a rational interpolation framework based on the rational projector $R_n^m$ recently introduced in \cite{Th-Barel}. Its application within the classical projection-collocation scheme naturally yields a Rational Collocation (RC) method that is stable, well conditioned and convergent, but requires the computation of the integral operator at the rational basis functions. By using de la Vall\'ee Poussin interpolation at the first kind Chebyshev zeros, we constructed a fully discrete method, called Rational Discrete Collocation (RDC) method, which maintains the  stability and convergence properties proved for RC method. As an alternative discretization strategy, we also considered Gaussian quadrature to approximate the integral operator obtaining a method, referred to as the Modified Nyström (MN) method, that exhibits the same linear system as the classical Nyström method but a different rational reconstruction of the solution. In the numerical experiments, both RDC and MN methods have been compared with the classical Nyström method (based on Gauss quadrature) and some of its recent variants (based on anti Gauss [], averaged Gauss [], and weighted averaged Gauss [] quadrature), addressing the challenging cases of discontinuous or highly oscillating integrand. It turns out that the approximate solution provided by MN method is almost comparable with that computed by the Nyström type methods, but it has the advantage of requiring only some sampling of the known terms of the equation. On the contrary, a superior performance has been observed for RDC method that consistently achieves higher accuracy than the other methods in the "difficult" cases we considered. 

Future work will investigate the use of the proposed rational projector in other numerical frameworks, including deferred correction methods \cite{SPIDeC_Mario} for differential problems and direct quadrature schemes for Volterra integral \cite{M_Volterra} and integro-differential \cite{M_Volterra_Integro_Diff} equations. Extensions of the present framework to  multidimensional integral equations \cite{Mirzaei_2021}, as well as to the non linear case, also deserve further investigation.

\section*{Acknowledgments}

This work has been partially supported by the Gruppo Nazionale Calcolo Scientifico-Istituto Nazionale di Alta Matematica (GNCS-INdAM) 
and by the PRIN 2022 PNRR no. P20229RMLB financed by the European Union - NextGeneration EU and by the Italian Ministry of University and Research (MUR).

This research has been accomplished within the RITA ``Research ITalian network on Approximation'', the UMI Group TAA ``Approximation Theory and Applications'', and the SIMAI Activity Group ANA$\&$A ``Numerical and Analytical Approximation of Data and Functions with Applications''.

The authors are grateful to Professor Luisa Fermo for generously sharing the codes used in the implementation of the Nystr{\"o}m variants.

\subsection*{Memberships}

All the  authors are members of Gruppo Nazionale Calcolo Scientifico-Istituto Nazionale di Alta Matematica (GNCS-INdAM),  UMI-TAA \textquotedblleft Approximation Theory and Applications'' Research Group, RITA \textquotedblleft Research
 ITalian network on Approximation'' and SIMAI Activity Group ANA$\&$A \textquotedblleft Numerical and Analytical Approximation of Data and Functions with Applications''.

\section*{Conflict of interest}
The authors have no conflicts of interest to declare relevant to this article's content.

\bibliographystyle{plainnat}
\bibliography{reference}

\end{document}